\documentclass{article}

\usepackage{microtype}
\usepackage{graphicx}
\usepackage{subcaption}
\usepackage{booktabs} 
\usepackage{multirow}
\usepackage{amsmath}
\usepackage{tabularx}
\usepackage{bbm}
\usepackage[bb=boondox]{mathalfa}
\usepackage{makecell}
\usepackage{hyperref}
\usepackage{fancyhdr}
\usepackage{algorithm}

\usepackage[accepted]{icml2026}

\usepackage{amsmath}
\usepackage{amssymb}
\usepackage{mathtools}
\usepackage{amsthm}
\usepackage{multirow}
\usepackage{bm}
\usepackage{adjustbox}

\usepackage{algorithm}

\usepackage{microtype}
\usepackage[capitalize,noabbrev]{cleveref}

\theoremstyle{plain}
\newtheorem{theorem}{Theorem}[section]

\theoremstyle{definition}

\newtheorem{assumption}[theorem]{Assumption}
\theoremstyle{remark}

\usepackage[textsize=tiny]{todonotes}

\begin{document}
\icmltitlerunning{Bayesian Interpolating Neural Network (B-INN)}

\twocolumn[
\icmltitle{Bayesian Interpolating Neural Network (B-INN): a scalable and reliable Bayesian model for large-scale physical systems}





\icmlsetsymbol{equal}{*}

\begin{icmlauthorlist}
\icmlauthor{Chanwook Park}{equal,comp,hidenn}
\icmlauthor{Brian Kim}{equal,puremath,appmath}
\icmlauthor{Jiachen Guo}{hidenn,yyy}
\icmlauthor{Wing Kam Liu}{comp,hidenn}
\end{icmlauthorlist}

\icmlaffiliation{comp}{Department of Mechanical Engineering, Northwestern University, Evanston, Illinois, USA}
\icmlaffiliation{appmath}{Department of Engineering Sciences and Applied Mathematics, Northwestern University, Evanston, Illinois, USA}
\icmlaffiliation{puremath}{Department of Mathematics, Northwestern University, Evanston, Illinois, USA}
\icmlaffiliation{yyy}{Theoretical and Applied Mechanics Program, Northwestern University, Evanston, Illinois, USA}
\icmlaffiliation{hidenn}{HIDENN-AI, Evanston, Illinois, USA}

\icmlcorrespondingauthor{Wing Kam Liu}{w-liu@northwestern.edu}


\icmlkeywords{Machine Learning, ICML}

\vskip 0.3in
]



\printAffiliationsAndNotice{\icmlEqualContribution}

\begin{abstract}
Neural networks and machine learning models for uncertainty quantification suffer from limited scalability and poor reliability compared to their deterministic counterparts. 
In industry-scale active learning settings, where generating a single high-fidelity simulation may require days or weeks of computation and produce data volumes on the order of gigabytes, they quickly become impractical. This paper proposes a scalable and reliable Bayesian surrogate model, termed the Bayesian Interpolating Neural Network (B-INN). The B-INN combines high-order interpolation theory with tensor decomposition and alternating direction algorithm to enable effective dimensionality reduction without compromising predictive accuracy. We theoretically show that the function space of a B-INN is a subset of that of Gaussian processes, while its Bayesian inference exhibits linear complexity, $\mathcal{O}(N)$, with respect to the number of training samples. Numerical experiments demonstrate that B-INNs can be from 20 times to 10,000 times faster with a robust uncertainty estimation compared to Bayesian neural networks and Gaussian processes. These capabilities make B-INN a practical foundation for uncertainty-driven active learning in large-scale industrial simulations, where computational efficiency and robust uncertainty calibration are paramount.
\end{abstract}

\section{Introduction}

Uncertainty quantification is an essential requirement in surrogate modeling for modern engineering and scientific systems, where uncertainties arise from material variability, model inadequacy, and limited or noisy data. Statistical surrogate models are therefore critical for applications such as design optimization, inverse problems, and digital twins. In addition, active learning frameworks for large-scale numerical simulations rely on predictive uncertainty to guide adaptive sampling and reduce computational cost. These needs motivate the development of surrogate models that are both statistically rigorous and computationally efficient.

Gaussian processes are among the most widely used statistical surrogate models due to their nonparametric formulation and closed-form Bayesian inference \cite{mackay2003information}. They provide principled predictions of both mean and uncertainty and have been successfully applied in many engineering problems. However, Gaussian processes suffer from severe scalability limitations, particularly for large datasets and high-dimensional input spaces \cite{liu2020gaussian}. These challenges significantly limit their applicability to large-scale full-field surrogate modeling for physical simulations.

Recently, the Interpolating Neural Network (INN) was introduced as a scalable, reliable, and interpretable surrogate modeling framework that unifies interpolation theory, tensor decomposition, and neural networks \cite{park2025unifying, guo2025interpolating}. INNs construct approximation spaces using interpolation functions with strong numerical properties, leading to high accuracy and computational efficiency. By exploiting tensor decomposition, INNs achieve favorable scaling with respect to input dimension and have demonstrated superior performance over conventional neural networks in deterministic surrogate modeling of physical systems. However, the original INN formulation does not explicitly provide statistical inference or uncertainty quantification.

In this work, we propose Bayesian INN (B-INN) as a scalable statistical surrogate modeling framework that augments INNs with Bayesian inference. By borrowing ideas from classical Bayesian linear regression \cite{smith1973general}, Bayesian inference is performed in the INN coefficient space, producing analytical expressions for predictive uncertainty. We first prove the equivalence between B-INN and Gaussian processes in low-dimensional settings, both theoretically and numerically. We then extend the framework to high-dimensional problems using tensor decomposition and an alternating direction algorithm, enabling scalable uncertainty-aware surrogate modeling without sacrificing predictive accuracy. Compared to Gaussian processes that have the complexity of $\mathcal{O}(N^3)$ where $N$ is the data size, B-INN complexity is $\mathcal{O}(N)$ with proportional speed up with the number of parallel processors (i.e., $P$ parallel processors can accelerate by $P$ times). 

\section{Related work}

\subsection{Scalable Gaussian processes (GPs)}

Scalable Gaussian process (GP) methods have been broadly categorized into global and local approximation strategies to alleviate the cubic computational complexity of standard GPs with respect to the data size \cite{liu2020gaussian}. Global approximations aim to distill information from the entire dataset by constructing reduced representations of the kernel matrix \cite{chalupka2013framework} or low-rank inducing-point approximations \cite{quinonero2005unifying}, thus reducing the effective problem size and lowering the training complexity from $\mathcal{O}(N^3)$ to $\mathcal{O}(NM^2)$ where $N$ and $M$ are the size of the original data and inducing-point (i.e. subset) data, respectively. In contrast, local approximations follow a divide-and-conquer philosophy, partitioning data into smaller subsets and training multiple local GP experts whose predictions are later combined via mixture-of-experts formulations \cite{gramacy2008bayesian, gramacy2016lagp}. While these approaches significantly improve the scalability of GP, both classes rely on approximations that may discard a substantial portion of the original data dependencies. Therefore, the trade-off between computational efficiency and model fidelity remains a central challenge in scalable GP research.

\subsection{Bayesian neural networks (BNNs)}

Bayesian neural networks (BNNs) treat the network parameters as random variables and infer their posterior distribution from the data, allowing predictive uncertainty to be quantified alongside mean predictions. The fact that BNNs enable epistemic (model) uncertainty to be distinguished from aleatoric (data) uncertainty has made BNNs particularly appealing for active learning in surrogate modeling of physical simulations \cite{jospin2022hands}. Bayesian inference in these models is commonly carried out using either Markov Chain Monte Carlo (MCMC) methods \cite{bardenet2017markov} or variational inference \cite{blei2017variational}. MCMC provides asymptotically exact posterior samples, but its computational cost becomes prohibitive for high-dimensional problems with large neural networks and datasets. Variational inference addresses this limitation by recasting posterior inference as an optimization problem over a restricted family of distributions. However, it comes at the cost of approximation bias that may underpredict complex or multimodal posteriors. This poses a significant challenge in active learning scenarios where sampling decisions depend critically on reliable measures of model uncertainty.

\subsection{Diffusion models for uncertainty quantification}

Recent studies have demonstrated that denoising diffusion probabilistic models have been successfully applied to surrogate modeling of solid mechanics \cite{jadhav2023stressd}, thermal \cite{ogoke2024inexpensive, shi2025diffusion}, and fluid systems \cite{liu2024uncertainty}, where they generate full field predictions together with uncertainty estimates. However, the practical adoption of diffusion-based surrogates for large-scale active learning remains questionable. Training diffusion models typically requires a large number of high-fidelity simulations, which is incompatible with many engineering scenarios where each simulation is computationally expensive. Moreover, from the uncertainty produced by diffusion models, it is highly challenging to disentangle epistemic and aleatoric uncertainties. Active learning, by contrast, critically depends on well-calibrated epistemic uncertainty to guide adaptive sampling and requires frequent and fast retraining of surrogate models as new data are acquired. The high computing cost and training instability of diffusion models make such iterative retraining impractical, thereby limiting their effectiveness as a core engine for active learning in large-scale surrogate modeling.

\section{Bayesian Interpolating Neural Network (B-INN)}

This chapter will elucidate how the B-INN is formulated from an INN, and prove the equivalence between B-INNs and Gaussian processes (GPs).

\subsection{Review of INN}

An INN is a network represented as a summation of the product of the grid-wise interpolation functions and the corresponding interpolation values \cite{park2025unifying}. The baseline code of INN can be found from: \href{https://github.com/hachanook/pyinn}{https://github.com/hachanook/pyinn}. In one-dimension (1D), for example, assuming the input domain is discretized with $J$ grid points, it reads:
\begin{equation}
    y(x) = \sum_{j=1}^J \phi_j(x) w_j,
    \label{eq:INN_1D}
\end{equation}
where $\phi_j(x)$ and $w_j$ are the interpolation function and value at grid point $j$, respectively. There are several ways to construct the interpolation functions. One can choose piecewise linear (i.e., hat) functions if the underlying physics is highly localized, compactly supported nonlinear functions such as \cite{park2023convolution} to capture local nonlinearity, or even global functions such as Gaussian bases if there is a long-range dependence or low-frequency response.
\begin{figure}[h]
\centering
\includegraphics[width=1.0\linewidth]{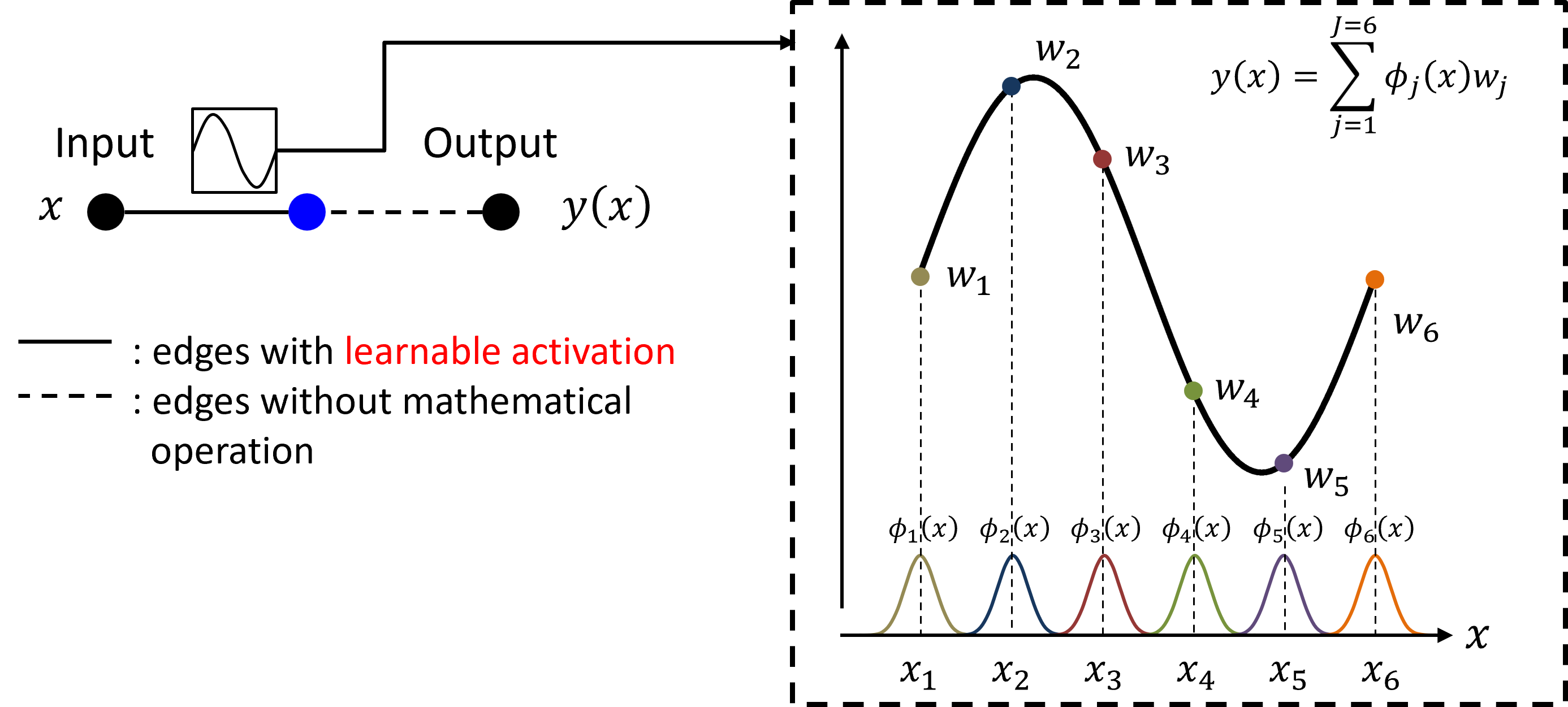}
\caption{INN architecture for a one-input and one-output system where the 1D input space is discretized with $J=6$ grid points. Gaussian functions are used for $\phi_{j}(x)$ \cite{park2025interpolating}. }
\label{fig:INN_1D}
\end{figure}

Figure \ref{fig:INN_1D} illustrates a simple one-input $(x)$ one-output $(y)$ INN architecture approximating a sine function. Between the input and output layers, there is only one hidden layer with one neuron, where the solid edge connecting from the input to the hidden neuron represents a learnable activation function parameterized with 6 interpolation values, $\{w_{j}\}_{j=1,\dots,6}$. These are the only trainable parameters of the INN, assuming that the 6 interpolation functions - $\{\phi_{j}(x)\}_{j=1,\dots,6}$ - that connect the interpolation values are fixed. Notice that although the original INN theory is open to adaptive interpolation functions \cite{park2025unifying}, we assume that they are fixed in this paper for simplicity.


The cost of constructing interpolation functions in high-dimensional space scales exponentially as the dimension grows. An INN in a high-dimensional input space thus adopts tensor decomposition (TD) theories such as Tucker decomposition \cite{tucker1963implications, tucker1966some} and CANDECOMP/PARAFAC (CP) decomposition \cite{harshman1970foundations, carroll1970analysis, kiers2000towards} to convert the high-dimensional interpolation into multiple 1D interpolations \cite{park2025unifying}. In this work, we adopt the CP decomposition that reads:
\begin{equation}
\begin{aligned}
     y &= \sum_{m=1}^M \prod_{d=1}^Df_d^{(m)}(x_d),\\
     &\text{where}\quad f_d^{(m)}(x_d) =\sum_{j=1}^{J_d} \phi_{d,j}(x_d) w_{d,j}^{(m)}
\end{aligned}
    \label{eq:INN_ND}
\end{equation}

Here, $M$ is the TD mode, $D$ is the input dimension, $x_d$ is the input variable of dimension $d$. The 1D function $f_d^{(m)}(x_d)$ now has dimensional index $d$ and mode index $m$ (compare this with Eq. \ref{eq:INN_1D}).

The multi-dimensional INN structure is described in Figure \ref{fig:INN_ND}. The forward pass consists of three steps: 1) interpolate each 1D input using the $f_d^{(m)}(x_d)$ function, 2) multiply the results of 1D interpolations at each mode, and 3) sum over $M$ modes in the output layer. 

\begin{figure}[h]
\centering
\includegraphics[width=1.0\linewidth]{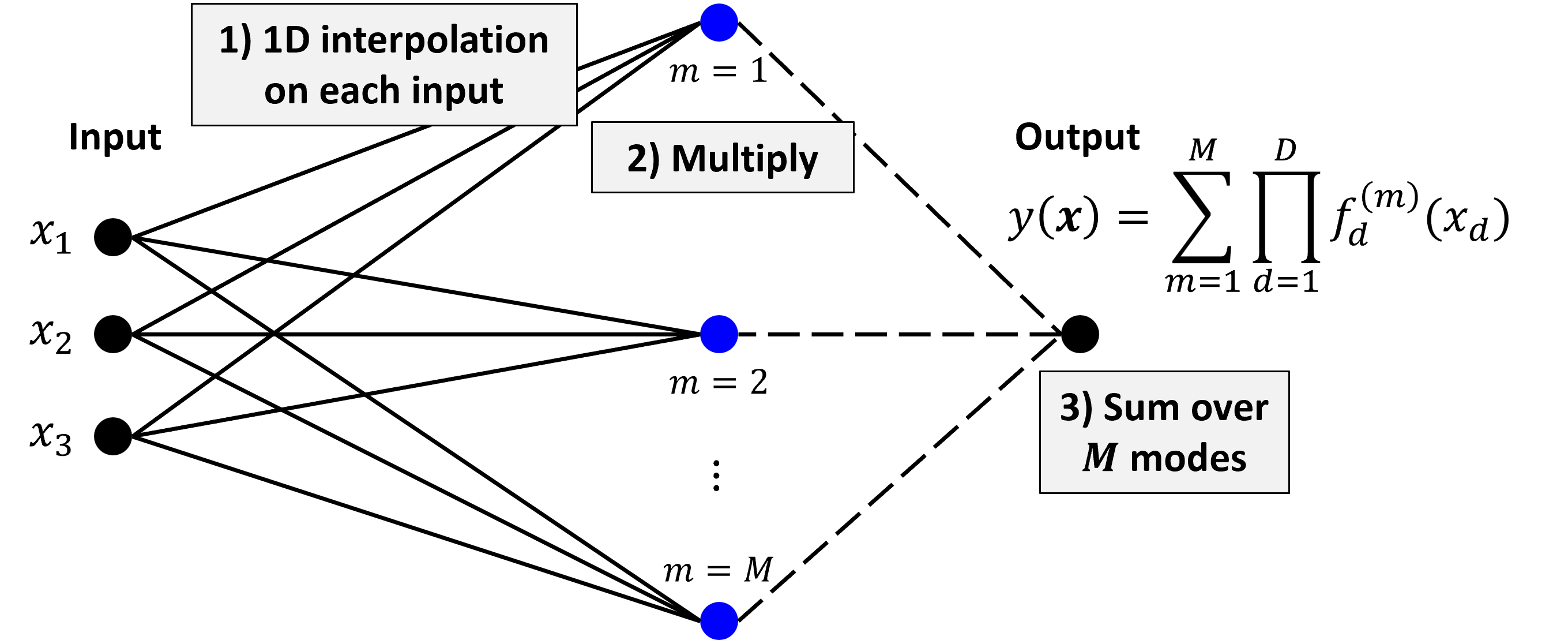}
\caption{INN architecture for a three-input and one-output system \cite{park2025interpolating}}
\label{fig:INN_ND}
\end{figure}

One may notice that the TD modes correspond to the hidden neurons in the network diagram, and the INN with CP TD only requires one hidden layer. This makes a critical distinction from multi-layer perceptrons (MLPs) and Kornogorov-Arnold Networks (KANs) \cite{liu2024kan}, which is summarized in Table \ref{tab:mlp-kan-inn}. Theoretically, all MLP, KAN, and INN have the so-called \textit{universal approximation} capabilities, although the TD theory achieves it in a discrete sense. That is, the CP TD can approximate an arbitrary tensor with a tiny error tolerance as the number of modes continues to increase \cite{kolda2009tensor}. However, because an INN requires only a single hidden layer, the number of trainable parameters scales linearly with the number of neurons, $M$, whereas both MLP and KAN scale quadratically. This makes an INN scalable. In addition, similar to KANs, INNs are interpretable, since we can engineer the interpolation functions to better capture the underlying physics. Readers may refer to \cite{park2025unifying} to understand how they are adapted during training.









\begin{table}[t]
\centering
\caption{Comparison across MLP, KAN, and INN architectures.}
\label{tab:mlp-kan-inn}
\renewcommand{\arraystretch}{1.25}
\resizebox{\columnwidth}{!}{%
\begin{tabular}{lccc}
\toprule
\textbf{Model} & \textbf{MLP} & \textbf{KAN} & \textbf{INN} \\
\midrule

\textbf{Theorem}
& \makecell[c]{Universal Approximation\\Theorem}
& \makecell[c]{Kolmogorov--Arnold\\Representation Theorem}
& \makecell[c]{Tensor Decomposition} \\
\midrule

\textbf{Network}
& Deep network
& Deep network
& Shallow network \\

\textbf{Depth}
& is preferred
& is needed
& is enough \\
\midrule

\textbf{Formula}
& \makecell[c]{$\displaystyle f_H \circ \cdots \circ f_1(\bm{x})$}
& \makecell[c]{$\displaystyle
\sum_{m=1}^{M} \Phi_m\!\left( \sum_{d=1}^{D} f_{m,d}(x_d) \right)$}
& \makecell[c]{$\displaystyle
\sum_{m=1}^{M} \prod_{d=1}^{D} f_d^{(m)}(x_d)$} \\
\midrule

\textbf{Trainable}
& $\mathcal{O}(H M^2)$
& $\mathcal{O}(H M^2 (J + k))$
& $MDJL$ \\

\textbf{Parameters}
& \multicolumn{3}{c}{
\makecell[c]{
$H$: hidden layers,\quad
$M$: hidden neurons,\quad
$D$: inputs,\quad
$L$: output \\
$k$: spline order of KAN,\quad
$J$: grid points of KAN and INN
}} \\
\bottomrule
\end{tabular}%
}
\end{table}


\subsection{Bayesian linear regression}

The Bayesian extension of an INN is simply done by viewing the INN trainable parameters (or interpolation values $w_{d,j}^{(m)}$) as random variables, following ideas similar to those introduced in seminal works in stochastic finite element methods \cite{liu1986random, liu1986probabilistic}. In the 1D case, this is mathematically identical to Bayesian linear regression (BLR) with nonlinear basis functions, $\phi_j(x)$. Thus, we review the BLR before proceeding to the B-INN in high-dimensional space.

BLR provides a probabilistic extension of classical linear regression by treating the weights as random variables. Classical linear regression has the form
\begin{equation}
    y = \boldsymbol{\phi}(\boldsymbol{x})^\top \boldsymbol{w} + \varepsilon,
    \label{eq:blr_model}
\end{equation}
where $\boldsymbol{\phi}(\boldsymbol{x}) \in \mathbb{R}^J$ is a vector of basis functions, $J$ denotes the number of basis functions, $\boldsymbol{w} \in \mathbb{R}^J$ is the weight vector, and $\varepsilon \sim \mathcal{N}(0,\sigma_n^2)$ represents independent Gaussian observation noise.

Given a dataset $\mathcal{D} = \{(\boldsymbol{x}^{(i)}, y^{(i)})\}_{i=1}^N$, where $N$ denotes the number of data points and $\boldsymbol{x}^{(i)} \in \mathbb{R}^D$ is the $i$-th input, we define the design matrix $\boldsymbol{X} \in \mathbb{R}^{N \times J}$, whose $i$-th row is the feature vector $\boldsymbol{\phi}(\boldsymbol{x}^{(i)})^\top$. In standard BLR, the design matrix coincides with the basis matrix
\begin{equation}
    \boldsymbol{\Phi}
    =
    \begin{bmatrix}
    \boldsymbol{\phi}(\boldsymbol{x}^{(1)})^\top \\
    \boldsymbol{\phi}(\boldsymbol{x}^{(2)})^\top \\
    \vdots \\
    \boldsymbol{\phi}(\boldsymbol{x}^{(N)})^\top
    \end{bmatrix}
    \in \mathbb{R}^{N \times J},
    \label{eq:design_matrix}
\end{equation}
and the target vector is $\boldsymbol{y} = (y^{(1)}, \ldots, y^{(N)})^\top\in\mathbb{R}^N$.

A Gaussian prior is placed on the weights,
\begin{equation}
    \boldsymbol{w} \sim \mathcal{N}(\boldsymbol{\mu}_w, \boldsymbol{\Sigma}_w),
    \label{eq:blr_prior}
\end{equation}
where $\boldsymbol{\mu}_w$ is the prior mean and $\boldsymbol{\Sigma}_w$ is the prior covariance matrix. Under the Gaussian likelihood induced by Eq.\eqref{eq:blr_model}, the posterior distribution of $\boldsymbol{w}$ remains Gaussian,
\begin{equation}
    \mathbb{P}(\boldsymbol{w} \mid \mathcal{D})
    =
    \mathcal{N}(\boldsymbol{\mu}_w^{\text{post}}, \boldsymbol{\Sigma}_w^{\text{post}}),
    \label{eq:blr_posterior}
\end{equation}
with posterior covariance and mean given by
\begin{equation}
    \label{eq:blr_posterior_eq}
    \begin{split}
        \boldsymbol{\Sigma}_w^{\text{post}}
        &=
        \left(
        \sigma_n^{-2} \boldsymbol{X}^\top \boldsymbol{X}
        +
        \boldsymbol{\Sigma}_w^{-1}
        \right)^{-1}, \\
        \boldsymbol{\mu}_w^{\text{post}}
        &=
        \boldsymbol{\Sigma}_w^{\text{post}}
        \left(
        \boldsymbol{\Sigma}_w^{-1} \boldsymbol{\mu}_w
        +
        \sigma_n^{-2} \boldsymbol{X}^\top \boldsymbol{y}
        \right).
    \end{split}
\end{equation}

Given a new input $\boldsymbol{x}^* \in \mathbb{R}^D$, the predictive distribution of the output $y^*$ is obtained by marginalizing over the posterior uncertainty in $\boldsymbol{w}$,
\begin{align}
    \mathbb{P}(y^* \mid \boldsymbol{x}^*, \mathcal{D})
    &= \mathcal{N}\!\left(\mu_*(\boldsymbol{x}^*),\, \sigma_*^2(\boldsymbol{x}^*)\right), \\
    \mu_*(\boldsymbol{x}^*)
    &= \boldsymbol{\phi}(\boldsymbol{x}^*)^\top \boldsymbol{\mu}_w^{\text{post}}, \\
    \sigma_*^2(\boldsymbol{x}^*)
    &= \sigma_n^2
    + \boldsymbol{\phi}(\boldsymbol{x}^*)^\top
    \boldsymbol{\Sigma}_w^{\text{post}}
    \boldsymbol{\phi}(\boldsymbol{x}^*).
\end{align}

The predictive uncertainty decomposes into an epistemic component arising from uncertainty in the model weights and an aleatoric component due to observation noise. This closed-form Bayesian inference framework forms the probabilistic backbone for B-INN.

\subsection{Bayesian inference of INNs}
B-INN follows the tensor-decomposition structure of the INN in Eq.\eqref{eq:INN_ND}. In each dimension, 
the one-dimensional input $x_d$ is expanded in the $J$-dimensional Gaussian basis space $\{\phi_{d,j}^{(m)}(x_d)\}_{j=1}^J$ and the weights are treated as random variables. Bayesian inference of the B-INN inherits ideas from the alternating least squares (ALS) formulation for efficiency and
scalability. Unlike classical ALS, which trains tensors on a discrete domain, the B-INN is trained on a continuous input domain by reformulating each ALS block update as a Bayesian linear regression
problem.

Let $D\in\mathbb{N}$ be the input dimension, and $M \in \mathbb{N}$ the number of modes. Given training inputs
$\{\boldsymbol{x}^{(i)}\}_{i=1}^N \subset \mathbb{R}^D$ with $\boldsymbol{x}^{(i)}=(x^{(i)}_1,\dots,x^{(i)}_D)$, let
$u:\mathbb{R}^D\to\mathbb{R}$ denote the target function and define the target vector
\[
    \mathbf{u}
    :=
    \begin{bmatrix}
        u(\boldsymbol{x}^{(1)})\\
        \vdots\\
        u(\boldsymbol{x}^{(N)})
    \end{bmatrix}
    \in\mathbb{R}^N.
\]
For each dimension $d\in\{1,\dots,D\}$, fix $J_d\in\mathbb{N}$ Gaussian basis functions with
centers $\{c_{d,j}\}_{j=1}^{J_d}\subset\mathbb{R}$ and length scale $l_d>0$. Define
\begin{equation}
    \label{eq:rbf_basis_scalar}
    \phi_{d,j}(x_d)
    := \exp\!\left(-\frac{(x_d-c_{d,j})^2}{2l_d^2}\right),
    \qquad j=1,\dots,J_d,
\end{equation}
For each training input $x_d^{(i)},$ define the basis evaluation vector
\begin{equation}
    \label{eq:basis_evaluation_vector}
    \Phi_d^{(i)}
:=
\begin{bmatrix}
\phi_{d,1}(x_d^{(i)})\\
\vdots\\
\phi_{d,J_d}(x_d^{(i)})
\end{bmatrix}
\in\mathbb{R}^{J_d}.
\end{equation}

For each mode $m\in\{1,\dots,M\}$ and dimension $d\in\{1,\dots,D\}$, the weights are distributed with a Gaussian prior:
\begin{equation}
    \label{eq:prior_wdm_generalD}
    w_d^{(m)} \sim \mathcal{N}\!\big(\mu_d^{(m)},\,\sigma_w^2 I_{J_d}\big)
\end{equation}
where $\mu_d^{(m)}\in\mathbb{R}^{J_d}$ and $\sigma_w^2>0$. Let
\begin{equation}
    \label{eq:weight_vector}
    w_d
    :=
    \begin{bmatrix}
        w_d^{(1)}\\
        \vdots\\
        w_d^{(M)}
    \end{bmatrix}
    \in\mathbb{R}^{M J_d},
    \qquad
    \mu_d
    :=
    \begin{bmatrix}
        \mu_d^{(1)}\\
        \vdots\\
        \mu_d^{(M)}
    \end{bmatrix}
    \in\mathbb{R}^{M J_d}.
\end{equation}

To update the weights in dimension $d$, we freeze the contributions in all other dimensions
$\ell\neq d$ at fixed values. Specifically,
for each $\ell\neq d$ and mode $m$, set
\[
    \widetilde w_\ell^{(m)}
    =
    \begin{cases}
        \mu_\ell^{(m)}, & \text{(initialization)},\\
        \widetilde m_\ell^{(m)}, & \text{(most recent posterior mean)}.
    \end{cases}
\]
Define the corresponding frozen factors from fixed dimensions
\[
    \widetilde f_\ell^{(m)}(x_\ell^{(i)})
    :=
    (\Phi_\ell^{(i)})^\top \widetilde w_\ell^{(m)},
    \qquad \ell\neq d
\]
For each training input $\boldsymbol{x}^{(i)}$, we define the contributions from all dimensions $\ell \neq d$ by

\begin{align}
    &g_d^{(m)}(x_{-d}^{(i)})
    :=
    \prod_{\substack{\ell=1\\ \ell\neq d}}^{D} \widetilde f_\ell^{(m)}(x_\ell^{(i)}), \\
    &g_d^{(i)}
    :=
    \begin{bmatrix}
        g_d^{(1)}(x_{-d}^{(i)})\\
        \vdots\\
        g_d^{(M)}(x_{-d}^{(i)})
    \end{bmatrix}
    \in\mathbb{R}^M .
\end{align}

Then for the $i$-th training input,
\begin{align}
    y(x^{(i)})
    &=
    \sum_{m=1}^M
    \Big((\Phi_d^{(i)})^\top w_d^{(m)}\Big)\,g_d^{(m)}(x_{-d}^{(i)}) \\
    &=
    \big(g_d^{(i)\top}\otimes (\Phi_d^{(i)})^\top\big)\,w_d.
\end{align}
Therefore, the design matrix for the $d$-th block update is
\begin{equation}
    \label{eq:Xd_kronecker_row}
    X_d
    :=
    \begin{bmatrix}
        g_d^{(1)\top}\otimes (\Phi_d^{(1)})^\top\\
        \vdots\\
        g_d^{(N)\top}\otimes (\Phi_d^{(N)})^\top
    \end{bmatrix}
    \in\mathbb{R}^{N\times (M J_d)}.
\end{equation}

With Gaussian observation noise, we model
\[
    \mathbf{u} = X_d\,w_d + \varepsilon,
    \qquad
    \varepsilon\sim\mathcal{N}(0,\sigma_n^2 I_N),
\]
where $\sigma_n^2>0$ is the noise variance. Conditional on the frozen contributions from the
other dimensions, this is exactly a Bayesian linear regression problem.

Combining the Gaussian likelihood with the prior (Eq.\eqref{eq:prior_wdm_generalD}), the posterior
$p(w_d\mid X_d,\mathbf{u})$ is Gaussian,
\[
    w_d \mid X_d,\mathbf{u} \sim \mathcal{N}\!\left(\widetilde m_d,\;\Sigma_{d}^{\text{post}}\right),
\]
where the posterior covariance $\Sigma_{d}^{\text{post}}$ and posterior mean $\widetilde m_d$ are given by
Eq.\eqref{eq:blr_posterior_eq} (with $\boldsymbol{X}=X_d$ and $\boldsymbol{y}=\mathbf{u}$).

The posterior mean $\widetilde m_d\in\mathbb{R}^{M J_d}$ can be reshaped into mode-wise means
$\{\widetilde m_d^{(m)}\}_{m=1}^M$ with $\widetilde m_d^{(m)}\in\mathbb{R}^{J_d}$, which we take
as the updated weights for dimension $d$.

We update each dimension $d=1,\dots,D$ in turn by rebuilding $X_d$ from $\{g_d^{(i)}\}_{i=1}^N$ and setting $\widetilde w_d \leftarrow \widetilde m_d$ via
Eq.\eqref{eq:blr_posterior_eq}. Repeating these alternating Bayesian linear regression updates
for a fixed number of iterations yields the Bayesian inference procedure for B-INN.

\subsection{Equivalence between GP and B-INN}
\label{subsec:Equivalence_between_GP_and_B-INN}

We can observe that if $D=1$, the B-INN function in Eq.\eqref{eq:INN_ND} collapses to Bayesian linear regression.
From Eq.\eqref{eq:weight_vector}, define
\[
 \mathbf{w} := \sum_{m=1}^M w_1^{(m)}, \quad \mathbf{w}\in\mathbb{R}^{J_1}.
\]
We retain the tensor decomposition mode number $M$ for generality, although it is typically set to $M=1$ for 1D problems.

Using Eq.\eqref{eq:basis_evaluation_vector} and $\mathbf{w}$, we can rewrite the 1D B-INN function
\begin{equation}
y(x_1) = \sum_{m=1}^M \sum_{j=1}^{J_1}\phi_{1,j}(x_1)\,w_{1,j}^{(m)} = \Phi_1(x_1)^\top \mathbf{w}.    
\end{equation}
Notice that the dimension index is added as a subscript to the input: $x_1$. Then $y(x_1)$ is linear in the weights $\mathbf{w}$, and with a Gaussian prior on $\mathbf{w}$ and Gaussian observation noise, the $D=1$ case reduces to a Bayesian linear regression model.

\begin{assumption}[Independent Gaussian weights]\label{assump:weights_main}
For every $d\in\{1,\dots,D\}$, $j\in\{1,\dots,J_d\}$, and mode $m\in\{1,\dots,M\}$,
\[
w_{d,j}^{(m)} \stackrel{\mathrm{i.i.d.}}{\sim} \mathcal{N}(0,\sigma_w^2),
\]
independent across all indices $(d,j,m)$.
\end{assumption}

\begin{theorem}[As $M\to\infty$, the normalized B-INN prior converges to a GP prior in finite-dimensional distributions]
\label{thm:prior_gp_main}
Assume Assumption~\ref{assump:weights_main}. 
Let $y^{(M)}$ denote the B-INN output $y$ in Eq.~\eqref{eq:INN_ND} when the number of modes is $M$, and define the normalization process
\begin{equation}\label{eq:binn_normalized}
\bar y^{(M)}(x):=\frac{1}{\sqrt{M}}\,y^{(M)}(x).
\end{equation}
Then, as $M\to\infty$, $\bar y^{(M)}$ converges in finite-dimensional distributions to a centered Gaussian process
\[
g \sim \text{GP}(0,k),
\]
with kernel
\[
k(x,x')
=
\prod_{d=1}^D
\left(
\sigma_w^2\sum_{j=1}^{J_d}\phi_{d,j}(x_d)\phi_{d,j}(x_d')
\right),
\]
where $\phi_{d.j}$ is the standard RBF basis function.
\end{theorem}

\begin{theorem}[As $M\to\infty$, the normalized B-INN posterior at finitely many test points converges weakly to the GP regression posterior]
\label{thm:postpred_gp} 
Assume ~\ref{assump:weights_main}. Let $X=\{x^{(i)}\}_{i=1}^n\subset\mathbb{R}^D$ be training inputs and let
$X_*=\{x_*^{(j)}\}_{j=1}^{n_*}\subset\mathbb{R}^D$ be test inputs. Fix observed data $D:=(X,\tilde y)$ with $\tilde y\in\mathbb{R}^n$.
For each $M$, define the posterior of $\bar y^{(M)}$ using the Gaussian likelihood
\[
\tilde y \mid \bar y^{(M)} \sim \mathcal{N}\!\big(\bar y^{(M)}(X),\,\sigma_n^2 I_n\big),
\qquad \sigma_n^2>0 \text{ fixed.}
\]
Let $g\sim\text{GP}(0,k)$ be the Gaussian-process limit from Theorem~\ref{thm:prior_gp_main}.
Then, as $M\to\infty$,
\[
\bar y^{(M)}(X_*) \mid D \;\Rightarrow\; g(X_*) \mid D
\quad\text{in distribution on }\mathbb{R}^{n_*}.
\]
Moreover, the limiting conditional distribution $g(X_*)\mid D$ is exactly the standard Gaussian-process regression posterior distribution at $X_*$
associated with the prior $g\sim\text{GP}(0,k)$ and noise variance $\sigma_n^2$ \cite{rasmussen2006gpml}.

\end{theorem}

The detailed proofs for these theorems can be found in Appendix \ref{appendix:proof}.

\subsection{Predictive mean and variance of B-INN}
Under the assumption \ref{assump:weights_main}, the predictive mean and variance can be derived as follows.

Let
\[
\boldsymbol{\Phi}_d :=
\begin{bmatrix}
\Phi_d^{(1)}\\
\vdots\\
\Phi_d^{(N)}
\end{bmatrix}
\in\mathbb{R}^{N \times J_d}.
\]

For a test input \(\boldsymbol{x}^+=(x_1^+,\dots,x_D^+)\),
\begin{align}
\mathbb{E}\!\left[y\right]
&=
\sum_{m=1}^{M}\prod_{d=1}^{D}
\boldsymbol{\Phi}_d(x_d^+)^T\,\widetilde{\boldsymbol{m}}_d^{(m)} .
\end{align}

Using the same independence assumption,
{\small
\begin{equation}
\label{eq:binn_pred_var}
\begin{aligned}
\mathrm{Var}\!\left[y\right]
= \sum_{m=1}^{M}
\Bigg(
&\prod_{d=1}^{D}\Big(
\mathbb{E}\!\left[f_d^{(m)}(x_d^+)\right]^2
- \mathrm{Var}\!\left[f_d^{(m)}(x_d^+)\right]
\Big) \\
&\;-\;
\prod_{d=1}^{D}\mathbb{E}\!\left[f_d^{(m)}(x_d^+)\right]^2
\Bigg).
\end{aligned}
\end{equation}
}

Finally, for each \(d\) and \(m\), we use the standard results from the Bayesian linear regression:
\[
\mathbb{E}\!\left[f_d^{(m)}(x_d^+)\right]
=
\Phi_d(x_d^+)^T\,\widetilde{m}_d^{(m)}
=
\sum_{j=1}^{J_d}\phi_{d,j}(x_d^+)\,\widetilde{m}_{d,j}^{(m)},
\]
\[
\begin{aligned}
\text{and} \quad\mathrm{Var}\!\left[f_d^{(m)}(x_d^+)\right]
&=
\Phi_d(x_d^+)^T\,V_d^{(m)} \Phi_d(x_d^+).
\end{aligned}
\]

\section{Results}
The numerical experiments consist of two parts. We first start with regression benchmarks with uncertainty quantification and compare B-INNs with Gaussian processes (GPs) and Bayesian neural networks (BNNs). Then we extend our discussion to active learning and demonstrate that B-INNs are significantly faster and more reliable than BNNs for high-dimensional physical surrogate modeling problems.

\subsection{Regression with uncertainty quantification}
\label{subsec:regression_with_UQ}

\subsubsection{GP, BNN, and B-INN in 1D input space}

As discussed in Section \ref{subsec:Equivalence_between_GP_and_B-INN}, the function space of B-INN is a subset of that of GP, while they become identical for an infinite number of TD modes (for multi-dimension). To validate this, let us consider a synthetic 1D function: $y = \sin(3x) + 0.3\cos(9x) + \epsilon$, where $\epsilon \sim \mathcal{N}(0, 0.04)$. From this function, 60 training and 200 test data points are randomly generated in $x \in [-1, 1]$.

\begin{figure}[h]
\centering
\includegraphics[width=1.0\linewidth]{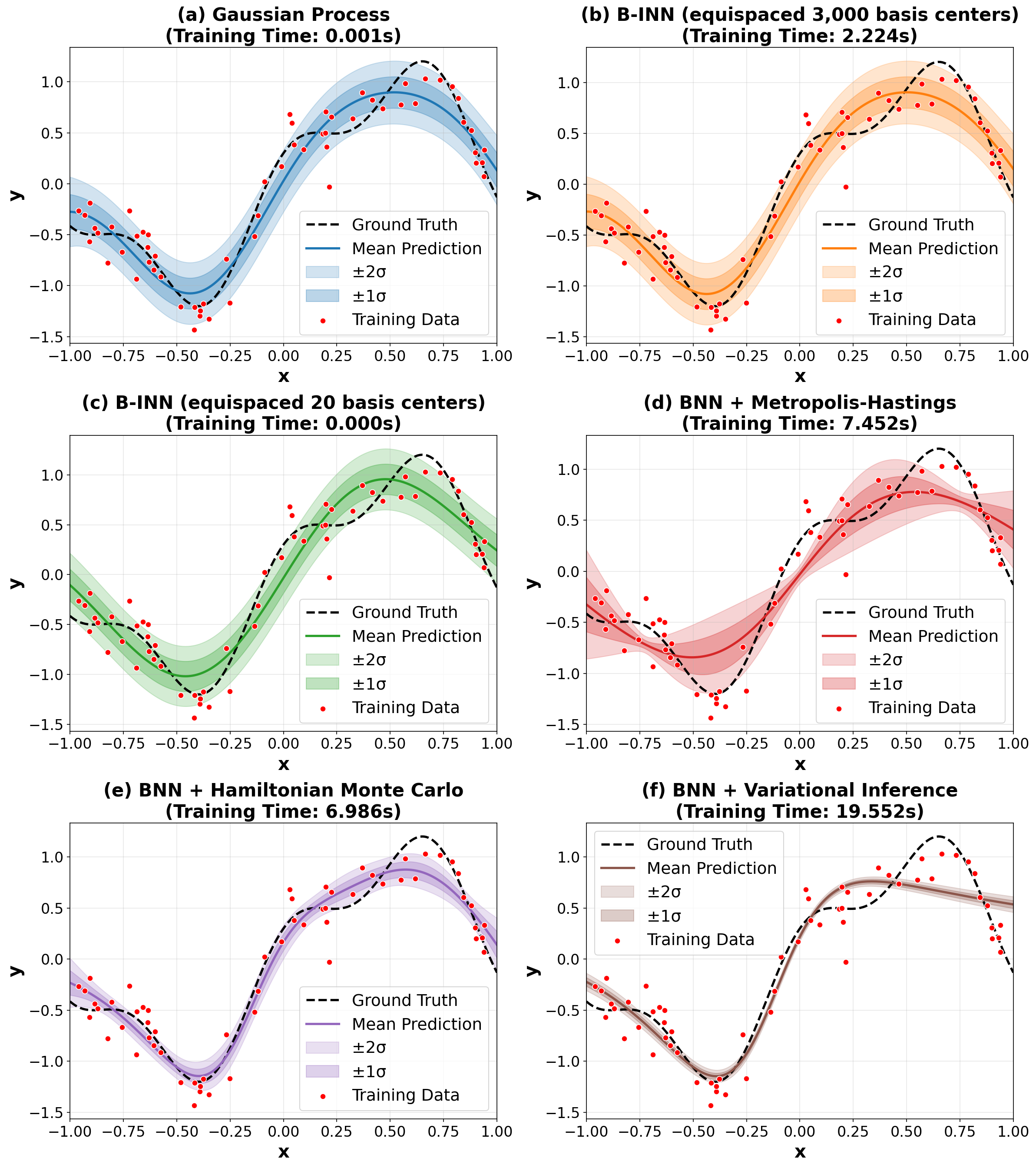}
\caption{Regression results of GP, B-INNs, and BNNs. Model hyperparameters can be found in Appendix \ref{appendix:model_hyperparameters}.}
\label{fig:1D_wo_gap_predictions_comparison}
\end{figure}

Figure \ref{fig:1D_wo_gap_predictions_comparison} compares 6 models: (a) Gaussian process, (b) B-INN with equispaced 3,000 basis functions, (c) B-INN with equispaced 20 basis functions, (d) BNN with Metropolis Hastings (MH), (e) BNN with Hamiltonian Monte Carlo (HMC), and (f) BNN with Variational Inference (VI). The first two results are nearly the same, while the maximum absolute difference of test RMSE mean and standard deviation among the GP (a) and the B-INN (b) are $1.78\times10^{-2}$ and $1.14\times10^{-3}$, respectively. In Figure \ref{fig:1D_wo_gap_predictions_comparison} (c), having a much smaller number of basis functions $(J=20 \ll 3,000)$ does not deteriorate predictive accuracy and uncertainty measure, although there is a dramatic reduction in operation count of matrix inversion: from $\mathcal{O}(3,000^3)$ of (b) to $\mathcal{O}(60^3)$ of (a) to $\mathcal{O}(20^3)$ of (c). Therefore, a B-INN with $J<N$ becomes substantially more efficient than the GP when the data set is huge.

Compared to the GP (a) and B-INNs (b, c), Monte Carlo-based BNNs (d, e) and the BNN with VI (f) perform poorly in terms of predictive accuracy, uncertainty estimation, and training time. Furthermore, the (d, e) require storing multiple samples of training parameters, which becomes prohibitive when there are billions of neural network parameters and thousands of samples. 

We conducted a scaling analysis of the same problem by progressively increasing the number of training data points $(N)$ from 160 to 1.3 million (see Figure \ref{fig:1D_wo_gap_data_scaling}(a)). When all other hyperparameters are fixed (see Appendix \ref{appendix:model_hyperparameters} for details), the training time of the GP scales cubically with $N$, whereas all other models exhibit linear scaling. We also observe that the B-INN is approximately $10^2$ times faster than the fastest BNN (with MH) and $10^4$ times faster than the slowest BNN (with HMC). Even with 1.3 million data points, the B-INN training took only one second. 

\begin{figure}[h]
\centering
\includegraphics[width=0.98\linewidth]{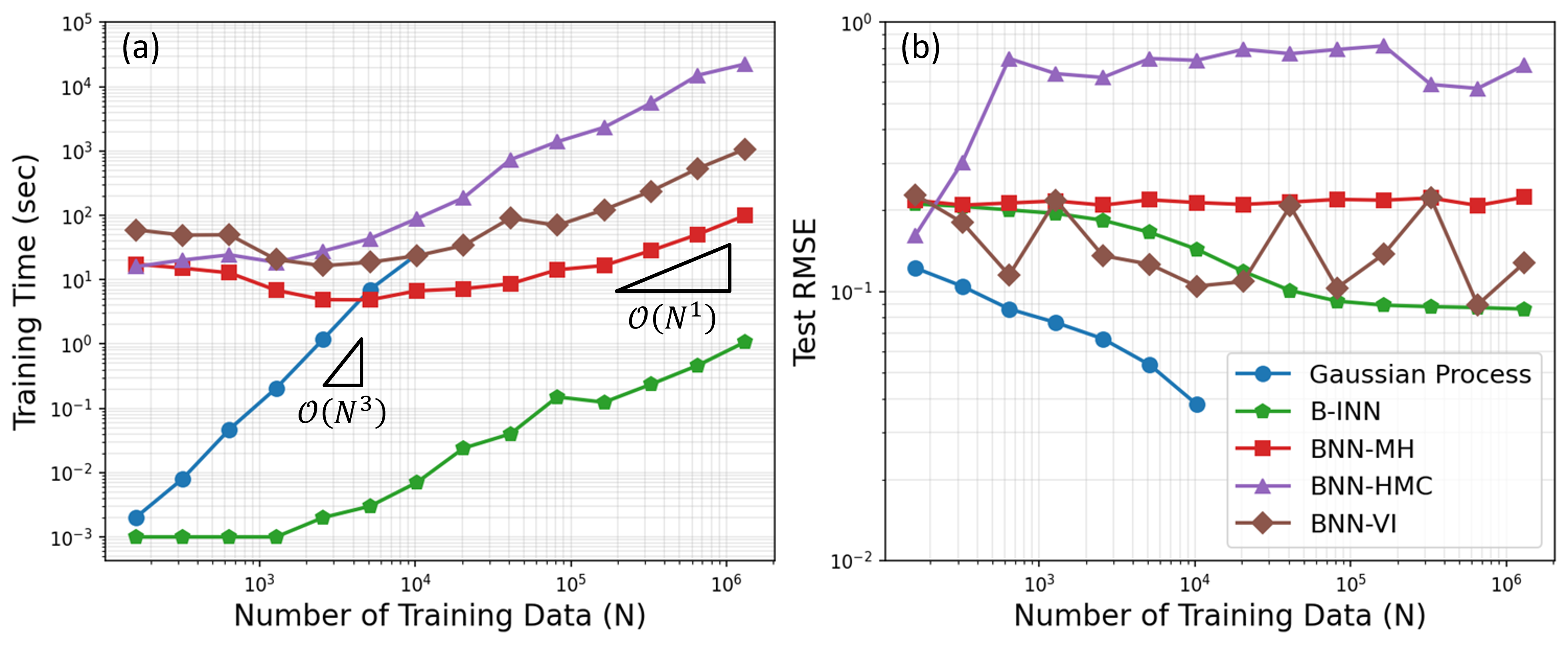}
\caption{Training time (a) and test RMSE (b) as the number of training samples increases.}
\label{fig:1D_wo_gap_data_scaling}
\end{figure}

The test root mean squared error (RMSE) is reported in Figure \ref{fig:1D_wo_gap_data_scaling}(b). While the GP error was the lowest, it cannot be trained on more than 10,000 data points on a desktop machine with 32GB of RAM. BNN predictions are generally less accurate and exhibit greater variability than those of the B-INN.
This 1D benchmark demonstrates why the B-INN can serve as a scalable and reliable alternative to GPs and BNNs.

\subsubsection{Learning open source aerodynamics dataset - BlendedNet}
\label{subsec:learn_Blendednet}
We then performed a similar analysis on a high-dimensional open-source aerodynamics data set called BlendedNet \cite{sung2025blendednet}. The original data size is roughly 40GB for the training set and 4GB for the test set when they are mounted on the CPU memory. Due to the scalability limitation of BNNs, we subsampled from the original dataset such that the train/test set sizes are roughly 80/8 MB (see Appendix \ref{appendix:blendednet}). In the current data, there are 7 inputs (three for $x,y,z$ coordinates on a plane surface and four flight conditions - altitude, Reynolds number, Mach number, and angle of attack) and 4 outputs (aerodynamic coefficients - $C_p, C_{f_x}, C_{f_y}, C_{f_z}$).

Table \ref{tab:blendednet_rmse_comparison} compares the predictive accuracy and training cost of B-INN against BNN-HMC and BNN-VI. Across all target quantities, B-INN consistently achieves the lowest or comparable training and test RMSE (see Figure \ref{fig:blendednet} for visual illustration), indicating superior predictive accuracy and generalization performance. Among these models, the BNN-VI shows the largest degradation in accuracy. From a computational standpoint, B-INN is significantly more efficient, requiring only about 20–30 seconds for training, whereas BNN-based models require approximately 550–600 seconds. Overall, B-INN is over 20 times faster than BNNs while delivering the most accurate predictions, underscoring its suitability for large-scale uncertainty-aware surrogate modeling and time-critical active learning workflows.

\begin{table}[h]
\centering
\caption{RMSE and training time comparison across models for BlendedNet Dataset}
\label{tab:blendednet_rmse_comparison}
\small
\renewcommand{\arraystretch}{1.15}
\resizebox{\columnwidth}{!}{%
\begin{tabular}{lcc|cc|cc}
\hline
\textbf{Metric}
& \multicolumn{2}{c|}{\textbf{B-INN}}
& \multicolumn{2}{c|}{\textbf{BNN-HMC}}
& \multicolumn{2}{c}{\textbf{BNN-VI}} \\
\cline{2-7}
& \textbf{Train} & \textbf{Time}
& \textbf{Train} & \textbf{Time}
& \textbf{Train} & \textbf{Time} \\
& \textbf{/Test} &  \textbf{(sec)}
& \textbf{/Test} &  \textbf{(sec)}
& \textbf{/Test} &  \textbf{(sec)} \\
\hline
\multirow{2}{*}{$C_p$}
& $3.32\times10^{-2}$ & \multirow{2}{*}{\textit{29.75}}
& $3.49\times10^{-2}$ & \multirow{2}{*}{\textit{577.78}}
& $4.83\times10^{-2}$ & \multirow{2}{*}{\textit{544.85}} \\
& $4.64\times10^{-2}$ &
& $4.73\times10^{-2}$ &
& $8.32\times10^{-2}$ & \\

\multirow{2}{*}{$C_{f_x}$}
& $2.74\times10^{-2}$ & \multirow{2}{*}{\textit{21.36}}
& $2.80\times10^{-2}$ & \multirow{2}{*}{\textit{574.04}}
& $3.80\times10^{-2}$ & \multirow{2}{*}{\textit{562.72}} \\
& $2.74\times10^{-2}$ &
& $3.02\times10^{-2}$ &
& $7.29\times10^{-2}$ & \\

\multirow{2}{*}{$C_{f_y}$}
& $1.87\times10^{-2}$ & \multirow{2}{*}{\textit{26.53}}
& $1.85\times10^{-2}$ & \multirow{2}{*}{\textit{577.58}}
& $3.22\times10^{-2}$ & \multirow{2}{*}{\textit{578.8}} \\
& $1.74\times10^{-2}$ &
& $1.94\times10^{-2}$ &
& $7.88\times10^{-2}$ & \\

\multirow{2}{*}{$C_{f_z}$}
& $2.79\times10^{-2}$ & \multirow{2}{*}{\textit{27.12}}
& $2.78\times10^{-2}$ & \multirow{2}{*}{\textit{575.15}}
& $3.85\times10^{-2}$ & \multirow{2}{*}{\textit{611.92}} \\
& $3.03\times10^{-2}$ &
& $3.15\times10^{-2}$ &
& $8.35\times10^{-2}$ & \\
\hline
\end{tabular}
}
\end{table}

\begin{figure}[h]
\centering
\includegraphics[width=0.98\linewidth]{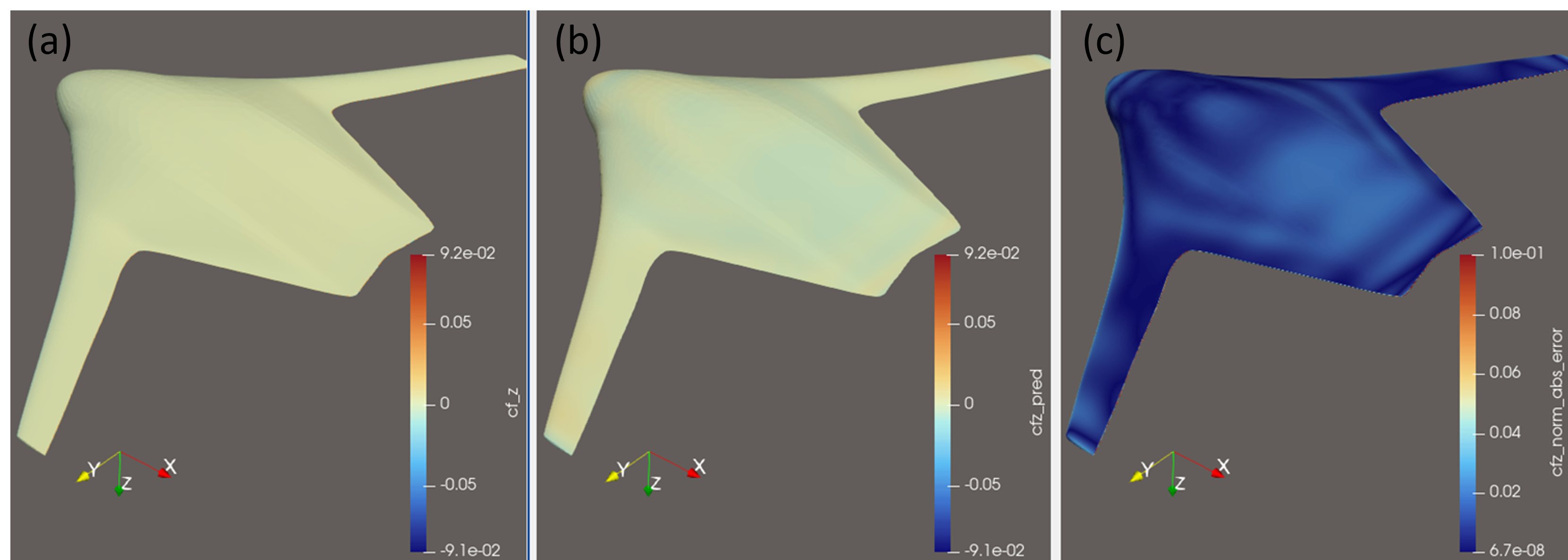}
\caption{Test data (a), B-INN prediction (b), and normalized absolute error (c) of force coefficient $(C_{f_z})$ of BlendedNet dataset. }
\label{fig:blendednet}
\end{figure}

\subsection{Active learning with B-INN}
\label{subsec:active_learning_with_B-INN}

Now, we explore how B-INNs can be used for active learning of high-dimensional physical problems. The active learning (AL) algorithm is provided in Appendix \ref{appendix:active_learning_algorithm}, where we referenced from \cite{mojumder2023linking}.

\subsubsection{Predicting a 3D space and 1D parameter Poisson equation solution with AL}
\label{subsec:3d_space_poisson_AL}


In this example, we aim to learn the solution from a parametric Poisson equation on the unit cube \(\Omega=[0,1]^3\) with homogeneous Dirichlet boundary conditions. For each parameter
\(p\in[0,1]\), we define \(u(\cdot;p)\) as the solution of
\begin{equation}
    \label{eq:poisson_fixed_pde}
    \begin{cases}
        -\Delta u(\boldsymbol{x};p) = f(\boldsymbol{x};p),
        & \boldsymbol{x}\in \Omega,\\[3pt]
        u(\boldsymbol{x};p)=0,
        & \boldsymbol{x}\in \partial\Omega,
    \end{cases}
\end{equation}
where \(\boldsymbol{x}=(x_1,x_2,x_3)\) and \(\Delta\) denotes the Laplacian with respect to \(\boldsymbol{x}\). Following the derivations provided in Appendix \ref{appendix:governing_equations_for_active_learning}, we obtain a fully analytic mapping \((\boldsymbol{x},p)\mapsto u(\boldsymbol{x};p)\), which we use to generate data.

We train and compare 3 different Models: B-INN, BNN-HMC, and BNN-VI. The results are summarized in Table~\ref{tab:rmse_poisson_comparison}, where the RMSE is reported. We conducted 20 active learning rounds. The candidate pool initially contained 100 parameter values, from which 8 were selected for the initial training set and 2 were held out as a fixed test set. For each spatial dimension, the grid contained 16 points. Additional details on the model configurations are provided in Appendix \ref{appendix:active_learning_and_model_hyperparameters}. For each model, we report the initial and final RMSE, and the best RMSE across all rounds. Overall, as also illustrated in Figure \ref{fig:Poisson_heat_RMSE}(a), B-INN achieves the lowest initial, minimum, and final RMSE, roughly three orders of magnitude lower than the best RMSE attained by BNN-HMC and BNN-VI. Unlike BNN-VI, which exhibits large fluctuations across rounds, B-INN shows consistent and monotonic error reduction, suggesting more stable learning dynamics under active acquisition.

\begin{table}[h]
\centering
\caption{Active learning RMSE comparison across models for 3D space 1D parameter Poisson equation}
\label{tab:rmse_poisson_comparison}
\small
\renewcommand{\arraystretch}{1.1}

\resizebox{\columnwidth}{!}{%
\begin{tabular}{lccc}
\hline
\textbf{Metric} & \textbf{B-INN} & \textbf{BNN-HMC} & \textbf{BNN-VI} \\
\hline
Init RMSE  & $2.55\times 10^{-5}$ & $1.45\times 10^{-3}$ & $1.39\times 10^{-2}$ \\
Best RMSE  & $2.12\times 10^{-6}$ & $3.75\times 10^{-4}$ & $1.04\times 10^{-3}$ \\
Final RMSE & $2.12\times 10^{-6}$ & $9.94\times 10^{-4}$ & $1.39\times 10^{-3}$ \\
\% Improvement    & 91.70\%             & 74.18\%              & 92.50\% \\

Total training time (sec)       & 261 & 8,233 & 19,321 \\
\hline
\end{tabular}
}
\end{table}

\begin{figure}[h]
\centering
\includegraphics[width=1.0\linewidth]{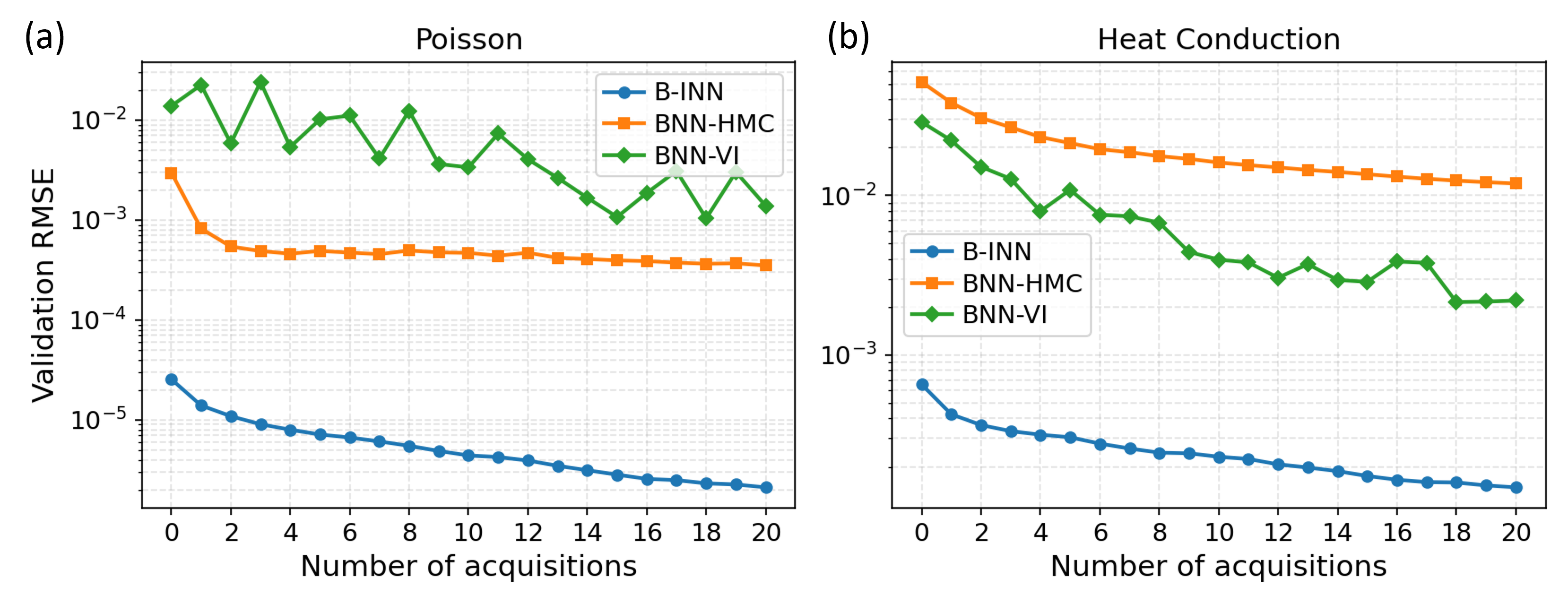}
\caption{Validation RMSE versus the number of active-learning acquisitions for the 3D spatial, 1D-parameter parametric Poisson problem (a) and 2D space, time, 2D parametric heat diffusion problem (b). Mode hyperparameters can be found in Appendix \ref{appendix:active_learning_and_model_hyperparameters}. }
\label{fig:Poisson_heat_RMSE}
\end{figure}

\subsubsection{Surrogate modeling of 2D space, time, and 2D parameter heat equation with AL}

We further increase the dimensionality of the problem. This time, we learn the solution to a time-dependent parametric heat conduction benchmark introduced in~\cite{guo2025interpolating}.
The goal is to learn the parametric solution operator mapping spatial location $(x,y)$, time $t$, thermal conductivity $k$,
and source power $P$ to the temperature field value $u$. The parametric PDE can be written as $(x,y,t,k,P)\mapsto u.$ The governing equation is provided in Appendix \ref{appendix:governing_equations_for_active_learning}. This benchmark problem is the largest in scale among those considered in this paper. After active learning, a total of 1.85 million data points were used for training.

Table \ref{tab:rmse_heat_comparison} and Figure \ref{fig:Poisson_heat_RMSE}(b) report RMSE over 20 active learning rounds. The spatial dimensions $(x,y)$ are discretized with 51 points each, time is discretized with 13 points, and each parameter $(k,P)$ is discretized with 11 points. This yields a candidate pool of $11\times 11=121$ parameter pairs. We use 35 $(k,P)$ pairs for the initial training set and hold out 10 $(k,P)$ pairs for validation. Details on the experimental setup and hyperparameters are provided in Appendix \ref{appendix:active_learning_and_model_hyperparameters}.

We can see that in Section \ref{subsec:3d_space_poisson_AL}, BNN-HMC outperforms BNN-VI on the smaller dataset. In contrast, for the larger-scale problem considered here, BNN-VI performs better, indicating that it scales more favorably with increasing data size. Nonetheless B-INNs consistently outperform both BNN baselines.
\begin{table}[h]
\centering
\caption{Active learning RMSE comparison across models for 2D space, time, 2D parameter heat diffusion equation}
\label{tab:rmse_heat_comparison}
\small
\renewcommand{\arraystretch}{1.1}

\resizebox{\columnwidth}{!}{%
\begin{tabular}{lccc}
\hline
\textbf{Metric} & \textbf{B-INN} & \textbf{BNN-HMC} & \textbf{BNN-VI} \\
\hline
Init RMSE  & $6.55\times 10^{-4}$ & $5.13\times 10^{-2}$ & $2.88\times 10^{-2}$ \\
Best RMSE  & $1.48\times 10^{-4}$ & $1.19\times 10^{-2}$ & $2.15\times 10^{-3}$ \\
Final RMSE & $1.48\times 10^{-4}$ & $1.19\times 10^{-2}$ & $2.19\times 10^{-3}$ \\
\% Improvement    & 77.37\%             & 74.18\%              & 92.50\% \\

Total training time (sec)       & 7,297 & 32,984 & 91,960 \\
\hline
\end{tabular}
}
\end{table}
\section{Conclusion}
We introduced a Bayesian interpolating neural network (B-INN) as a scalable and reliable alternative to Gaussian processes (GPs) and Bayesian neural networks (BNNs) for large-scale surrogate modeling with uncertainty quantification. Theoretically, we proved that the B-INN is a scalable and accurate subset of the GP. Across a range of numerical experiments, including high-dimensional active learning tasks, B-INN consistently outperformed GPs and BNNs, achieving speedups ranging from 20 to 10,000 times while providing comparable or improved predictive accuracy.

As with any early-stage methodology, several important research questions remain to be addressed. In particular, further investigation into the hyperparameter sensitivity of B-INN is warranted. For high-dimensional problems, we observe notable variations in training performance as key hyperparameters, such as the number of basis functions and characteristic length scales, are varied. This behavior is likely related to the alternating direction algorithm employed during Bayesian inference, which may introduce sensitivity to local optimization structure. One promising direction to mitigate this issue is to perform Bayesian inference directly in a global trainable parameter space, in a manner analogous to the BNN with variational inference.

Nevertheless, the theoretical analysis and numerical results presented in this paper demonstrate that B-INN provides a practical foundation for active learning in large-scale industrial simulations, enabling the efficient treatment of previously computationally intractable problems.
\section*{Impact Statement}

B-INN lowers the barrier to deploying principled uncertainty quantification and active learning in industry-grade physical surrogate modeling, including design optimization and digital twin development. More broadly, the proposed framework bridges classical statistical rigor and modern scalable learning, opening new opportunities for resource-efficient scientific AI.

\nocite{langley00}

\bibliography{example_paper}

@inproceedings{langley00,
 author    = {P. Langley},
 title     = {Crafting Papers on Machine Learning},
 year      = {2000},
 pages     = {1207--1216},
 editor    = {Pat Langley},
 booktitle     = {Proceedings of the 17th International Conference
              on Machine Learning (ICML 2000)},
 address   = {Stanford, CA},
 publisher = {Morgan Kaufmann}
}

@book{mackay2003information,
  title={Information theory, inference and learning algorithms},
  author={MacKay, David JC},
  year={2003},
  publisher={Cambridge university press}
}

@article{liu2020gaussian,
  title={When Gaussian process meets big data: A review of scalable GPs},
  author={Liu, Haitao and Ong, Yew-Soon and Shen, Xiaobo and Cai, Jianfei},
  journal={IEEE transactions on neural networks and learning systems},
  volume={31},
  number={11},
  pages={4405--4423},
  year={2020},
  publisher={IEEE}
}

@article{park2025unifying,
  title={Unifying machine learning and interpolation theory via interpolating neural networks},
  author={Park, Chanwook and Saha, Sourav and Guo, Jiachen and Zhang, Hantao and Xie, Xiaoyu and Bessa, Miguel A and Qian, Dong and Chen, Wei and Wanger, Gregory J and Cao, Jian and others},
  journal={Nature Communications},
  volume={16},
  number={1},
  pages={8753},
  year={2025},
  publisher={Nature Publishing Group UK London}
}

@InProceedings{guo2025interpolating,
  title = 	 {Interpolating Neural Network-Tensor Decomposition ({INN}-{TD}): a scalable and interpretable approach for large-scale physics-based problems},
  author =       {Guo, Jiachen and Xie, Xiaoyu and Park, Chanwook and Zhang, Hantao and Politis, Matthew J. and Domel, Gino and Liu, Wing Kam},
  booktitle = 	 {Proceedings of the 42nd International Conference on Machine Learning},
  pages = 	 {21138--21162},
  year = 	 {2025},
  volume = 	 {267},
  series = 	 {Proceedings of Machine Learning Research},
  month = 	 {13--19 Jul},
  publisher =    {PMLR},
  url = 	 {https://proceedings.mlr.press/v267/guo25p.html}
}

@article{smith1973general,
  title={A general Bayesian linear model},
  author={Smith, Adrian FM},
  journal={Journal of the Royal Statistical Society Series B: Statistical Methodology},
  volume={35},
  number={1},
  pages={67--75},
  year={1973},
  publisher={Oxford University Press}
}

@article{chalupka2013framework,
  title={A framework for evaluating approximation methods for Gaussian process regression},
  author={Chalupka, Krzysztof and Williams, Christopher KI and Murray, Iain},
  journal={The Journal of Machine Learning Research},
  volume={14},
  number={1},
  pages={333--350},
  year={2013},
  publisher={JMLR. org}
}

@article{quinonero2005unifying,
  title={A unifying view of sparse approximate Gaussian process regression},
  author={Quinonero-Candela, Joaquin and Rasmussen, Carl Edward},
  journal={Journal of machine learning research},
  volume={6},
  number={Dec},
  pages={1939--1959},
  year={2005}
}

@article{gramacy2008bayesian,
  title={Bayesian treed Gaussian process models with an application to computer modeling},
  author={Gramacy, Robert B and Lee, Herbert K H},
  journal={Journal of the American Statistical Association},
  volume={103},
  number={483},
  pages={1119--1130},
  year={2008},
  publisher={Taylor \& Francis}
}

@article{gramacy2016lagp,
  title={laGP: large-scale spatial modeling via local approximate Gaussian processes in R},
  author={Gramacy, Robert B},
  journal={Journal of Statistical Software},
  volume={72},
  pages={1--46},
  year={2016}
}

@article{jadhav2023stressd,
  title={StressD: 2D Stress estimation using denoising diffusion model},
  author={Jadhav, Yayati and Berthel, Joseph and Hu, Chunshan and Panat, Rahul and Beuth, Jack and Farimani, Amir Barati},
  journal={Computer Methods in Applied Mechanics and Engineering},
  volume={416},
  pages={116343},
  year={2023},
  publisher={Elsevier}
}

@article{ogoke2024inexpensive,
  title={Inexpensive high fidelity melt pool models in additive manufacturing using generative deep diffusion},
  author={Ogoke, Francis and Liu, Quanliang and Ajenifujah, Olabode and Myers, Alexander and Quirarte, Guadalupe and Malen, Jonathan and Beuth, Jack and Farimani, Amir Barati},
  journal={Materials \& Design},
  volume={245},
  pages={113181},
  year={2024},
  publisher={Elsevier}
}

@article{shi2025diffusion,
  title={Diffusion-based surrogate modeling and multi-fidelity calibration},
  author={Shi, Naichen and Yan, Hao and Guo, Shenghan and Al Kontar, Raed},
  journal={IEEE Transactions on Automation Science and Engineering},
  year={2025},
  publisher={IEEE}
}

@article{liu2024uncertainty,
  title={Uncertainty-aware surrogate models for airfoil flow simulations with denoising diffusion probabilistic models},
  author={Liu, Qiang and Thuerey, Nils},
  journal={AIAA Journal},
  volume={62},
  number={8},
  pages={2912--2933},
  year={2024},
  publisher={American Institute of Aeronautics and Astronautics}
}

@article{jospin2022hands,
  title={Hands-on Bayesian neural networks—A tutorial for deep learning users},
  author={Jospin, Laurent Valentin and Laga, Hamid and Boussaid, Farid and Buntine, Wray and Bennamoun, Mohammed},
  journal={IEEE Computational Intelligence Magazine},
  volume={17},
  number={2},
  pages={29--48},
  year={2022},
  publisher={IEEE}
}

@article{bardenet2017markov,
  title={On Markov chain Monte Carlo methods for tall data},
  author={Bardenet, R{\'e}mi and Doucet, Arnaud and Holmes, Chris},
  journal={Journal of Machine Learning Research},
  volume={18},
  number={47},
  pages={1--43},
  year={2017}
}

@article{blei2017variational,
  title={Variational inference: A review for statisticians},
  author={Blei, David M and Kucukelbir, Alp and McAuliffe, Jon D},
  journal={Journal of the American statistical Association},
  volume={112},
  number={518},
  pages={859--877},
  year={2017},
  publisher={Taylor \& Francis}
}

@article{park2023convolution,
  title={Convolution hierarchical deep-learning neural network (c-hidenn) with graphics processing unit (gpu) acceleration},
  author={Park, Chanwook and Lu, Ye and Saha, Sourav and Xue, Tianju and Guo, Jiachen and Mojumder, Satyajit and Apley, Daniel W and Wagner, Gregory J and Liu, Wing Kam},
  journal={Computational Mechanics},
  volume={72},
  number={2},
  pages={383--409},
  year={2023},
  publisher={Springer}
}

@article{tucker1966some,
  title={Some mathematical notes on three-mode factor analysis},
  author={Tucker, Ledyard R},
  journal={Psychometrika},
  volume={31},
  number={3},
  pages={279--311},
  year={1966},
  publisher={Springer}
}

@article{tucker1963implications,
  title={Implications of factor analysis of three-way matrices for measurement of change},
  author={Tucker, Ledyard R},
  journal={Problems in measuring change},
  volume={15},
  number={122-137},
  pages={3},
  year={1963},
  publisher={University of Wisconsin Press Madison}
}

@article{harshman1970foundations,
  title={Foundations of the PARAFAC procedure: Models and conditions for an “explanatory” multi-modal factor analysis},
  author={Harshman, Richard A and others},
  journal={UCLA working papers in phonetics},
  volume={16},
  number={1},
  pages={84},
  year={1970},
  publisher={Los Angeles, CA}
}

@article{carroll1970analysis,
  title={Analysis of individual differences in multidimensional scaling via an N-way generalization of “Eckart-Young” decomposition},
  author={Carroll, J Douglas and Chang, Jih-Jie},
  journal={Psychometrika},
  volume={35},
  number={3},
  pages={283--319},
  year={1970},
  publisher={Springer}
}

@article{kiers2000towards,
  title={Towards a standardized notation and terminology in multiway analysis},
  author={Kiers, Henk AL},
  journal={Journal of Chemometrics: A Journal of the Chemometrics Society},
  volume={14},
  number={3},
  pages={105--122},
  year={2000},
  publisher={Wiley Online Library}
}

@article{liu2024kan,
  title={Kan: Kolmogorov-arnold networks},
  author={Liu, Ziming and Wang, Yixuan and Vaidya, Sachin and Ruehle, Fabian and Halverson, James and Solja{\v{c}}i{\'c}, Marin and Hou, Thomas Y and Tegmark, Max},
  journal={arXiv preprint arXiv:2404.19756},
  year={2024}
}

@article{kolda2009tensor,
  title={Tensor decompositions and applications},
  author={Kolda, Tamara G and Bader, Brett W},
  journal={SIAM review},
  volume={51},
  number={3},
  pages={455--500},
  year={2009},
  publisher={SIAM}
}

@inproceedings{sung2025blendednet,
  title={Blendednet: A blended wing body aircraft dataset and surrogate model for aerodynamic predictions},
  author={Sung, Nicholas and Spreizer, Steven and Elrefaie, Mohamed and Samuel, Kaira and Jones, Matthew C and Ahmed, Faez},
  booktitle={International Design Engineering Technical Conferences and Computers and Information in Engineering Conference},
  volume={89237},
  pages={V03BT03A049},
  year={2025},
  organization={American Society of Mechanical Engineers}
}

@article{hron2020exact,
  title={Exact posterior distributions of wide Bayesian neural networks},
  author={Hron, Ji\v{r}\'{i} and Bahri, Yasaman and Novak, Roman and Pennington, Jeffrey and Sohl-Dickstein, Jascha},
  journal={arXiv preprint arXiv:2006.10541},
  year={2020}
}

@book{rasmussen2006gpml,
  title={Gaussian Processes for Machine Learning},
  author={Rasmussen, Carl E. and Williams, Christopher K. I.},
  publisher={MIT Press},
  year={2006}
}

@phdthesis{park2025interpolating,
  title={Interpolating Neural Networks for the Next-Generation Predictive Scientific Artificial Intelligence},
  author={Park, Chanwook},
  year={2025},
  school={Northwestern University}
}

@article{mojumder2023linking,
  title={Linking process parameters with lack-of-fusion porosity for laser powder bed fusion metal additive manufacturing},
  author={Mojumder, Satyajit and Gan, Zhengtao and Li, Yangfan and Al Amin, Abdullah and Liu, Wing Kam},
  journal={Additive Manufacturing},
  volume={68},
  pages={103500},
  year={2023},
  publisher={Elsevier}
}

@article{liu1986random,
  title={Random field finite elements},
  author={Liu, Wing Kam and Belytschko, Ted and Mani, A8616800597},
  journal={International journal for numerical methods in engineering},
  volume={23},
  number={10},
  pages={1831--1845},
  year={1986},
  publisher={Wiley Online Library}
}

@article{liu1986probabilistic,
  title={Probabilistic finite elements for nonlinear structural dynamics},
  author={Liu, Wing Kam and Belytschko, Ted and Mani, A},
  journal={Computer Methods in Applied Mechanics and Engineering},
  volume={56},
  number={1},
  pages={61--81},
  year={1986},
  publisher={Elsevier}
}
\bibliographystyle{icml2026}

\newpage
\appendix
\onecolumn

\section{Proofs for section \ref{subsec:Equivalence_between_GP_and_B-INN}}
\label{appendix:proof}

\subsection{Proof of Theorem \ref{thm:prior_gp_main}}

Fix inputs $x^{(1)},\dots,x^{(n)}\in\mathbb{R}^D$. For each mode $m$, define
\[
U^{(m)}(x):=\prod_{d=1}^D f_d^{(m)}(x_d),
\qquad
Z^{(m)}:=\big(U^{(m)}(x^{(1)}),\dots,U^{(m)}(x^{(n)})\big)\in\mathbb{R}^n.
\]
Then the evaluation vector of the normalized process can be written as 
\[
\big(\bar{y}^{(M)}(x^{(1)}),\dots,\bar{y}^{(M)}(x^{(n)})\big)
=
\frac{1}{\sqrt{M}}\sum_{m=1}^M Z^{(m)}.
\]
\medskip
\noindent\textbf{Step 1: i.i.d.\ structure and finite second moments.}
By our assumption, the weights are independent and identically distributed across modes,
so $\{Z^{(m)}\}_{m\ge1}$ are i.i.d.\ in $\mathbb{R}^n$.

Let $d\in\{1,\dots,D\}$ and $u\in\mathbb{R}$. Since
$
f_d^{(m)}(u)=\sum_{j=1}^{J_d} w_{d,j}^{(m)}\phi_{d,j}(u)
$
is a finite linear combination of independent centered Gaussians,
\[
\mathbb{E}[f_d^{(m)}(u)]=0,
\qquad
\mathbb{E}\!\big[f_d^{(m)}(u)^2\big]
=
\sum_{r=1}^{J_d}\mathbb{E}\!\big[(w_{d,r}^{(m)})^2\big]\phi_{d,r}(u)^2
=
\sigma_w^2\sum_{r=1}^{J_d}\phi_{d,r}(u)^2
<\infty.
\]
Using independence across dimensions $d$, for any $x\in\mathbb{R}^D$,
\[
\mathbb{E}[U^{(m)}(x)]
=\prod_{d=1}^D \mathbb{E}[f_d^{(m)}(x_d)]
=0,
\qquad
\mathbb{E}\!\big[U^{(m)}(x)^2\big]
=\prod_{d=1}^D \mathbb{E}\!\big[f_d^{(m)}(x_d)^2\big]
<\infty.
\]
Hence $\mathbb{E}[Z^{(m)}]=0$ and
\(
\mathbb{E}\|Z^{(m)}\|^2=\sum_{i=1}^n \mathbb{E}[U^{(m)}(x^{(i)})^2]<\infty.
\)

\medskip
\noindent\textbf{Step 2: Multivariate CLT.}
By Step~1, $\{Z^{(m)}\}_{m\ge1}$ are i.i.d.\ with mean $0$ and finite second moment, so
\[
\frac{1}{\sqrt{M}}\sum_{m=1}^M Z^{(m)}
\Rightarrow
\mathcal{N}(0,\Sigma),
\qquad
\Sigma=\mathrm{Cov}(Z^{(m)}).
\]
Equivalently, for $1\le i,j\le n$,
\[
\Sigma_{ij}=\mathbb{E}\!\big[Z^{(m)}_i Z^{(m)}_j\big].
\]

\medskip
\noindent\textbf{Step 3: Define and compute the limiting kernel.}
Define the covariance kernel of the normalized process $\bar{y}^{(M)}$ by
\[
k(x,x'):=\mathbb{E}\!\big[\bar{y}^{(M)}(x)\bar{y}^{(M)}(x')\big],
\qquad x,x'\in\mathbb{R}^D.
\]

Then

\begin{align}
k(x,x')
&=\mathbb{E}\!\left[\left(\frac{1}{\sqrt{M}}\sum_{m=1}^M U^{(m)}(x)\right)
                \left(\frac{1}{\sqrt{M}}\sum_{\ell=1}^M U^{(\ell)}(x')\right)\right] \\
    &=\frac{1}{M}\sum_{m=1}^M\sum_{\ell=1}^M \mathbb{E}\!\big[U^{(m)}(x)\,U^{(\ell)}(x')\big].
\end{align}
For $m\neq \ell$, independence across distinct modes and $\mathbb{E}[U^{(m)}(x)]=0$ imply
\[
\mathbb{E}\!\big[U^{(m)}(x)\,U^{(\ell)}(x')\big]
=\mathbb{E}[U^{(m)}(x)]\,\mathbb{E}[U^{(\ell)}(x')]
=0,
\]
so all cross terms vanish and therefore
\begin{align}
k(x,x')
&=\frac{1}{M}\sum_{m=1}^M \mathbb{E}\!\big[U^{(m)}(x)\,U^{(m)}(x')\big]
=\mathbb{E}\!\big[U^{(m)}(x)\,U^{(m)}(x')\big].
\end{align}
by independence across dimensions $d,$ expanding $U^{(m)}(x)$ gives
\begin{align}
k(x,x')
&=\mathbb{E}\!\left[\prod_{d=1}^D f_d^{(m)}(x_d)\,f_d^{(m)}(x_d')\right]
=\prod_{d=1}^D \mathbb{E}\!\big[f_d^{(m)}(x_d)f_d^{(m)}(x_d')\big],
\end{align}
Finally, for each $d$ and $u,v\in\mathbb{R}$, expanding
\(
f_d^{(m)}(u)=\sum_{j=1}^{J_d} w_{d,j}^{(m)}\phi_{d,j}(u)
\)
gives
\begin{align}
\mathbb{E}\!\big[f_d^{(m)}(u)f_d^{(m)}(v)\big]
&=\mathbb{E}\!\left[\left(\sum_{j=1}^{J_d} w_{d,j}^{(m)}\phi_{d,j}(u)\right)
            \left(\sum_{s=1}^{J_d} w_{d,s}^{(m)}\phi_{d,s}(v)\right)\right] \\
&=\sum_{j,s=1}^{J_d}\mathbb{E}\!\big[w_{d,j}^{(m)}w_{d,s}^{(m)}\big]\phi_{d,j}(u)\phi_{d,s}(v)
=\sigma_w^2\sum_{j=1}^{J_d}\phi_{d,j}(u)\phi_{d,j}(v),
\end{align}
and therefore
\[
k(x,x')
=
\prod_{d=1}^D\left(\sigma_w^2\sum_{j=1}^{J_d}\phi_{d,j}(x_d)\phi_{d,j}(x_d')\right).
\]

Moreover, since $Z^{(m)}_i=U^{(m)}(x^{(i)})$, we have for $1\le i,j\le n$,
\[
\Sigma_{ij}
=\mathbb{E}\!\big[Z^{(m)}_i Z^{(m)}_j\big]
=\mathbb{E}\!\big[U^{(m)}(x^{(i)})\,U^{(m)}(x^{(j)})\big]
=k \!\big(x^{(i)},x^{(j)}\big),
\]
so the covariance matrix in the multivariate CLT is exactly the kernel matrix induced by $k$.

\medskip
\noindent\textbf{Step 4: Validity of $k$ as a covariance kernel.}
We show that $k$ is symmetric and positive semidefinite.
Symmetry is immediate since
\[
k(x,x')
=\mathbb{E}\!\big[U^{(m)}(x)\,U^{(m)}(x')\big]
=\mathbb{E}\!\big[U^{(m)}(x')\,U^{(m)}(x)\big]
=k(x',x).
\]
To prove positive semidefiniteness, fix any $n\in\mathbb{N}$, any inputs $x^{(1)},\dots,x^{(n)}\in\mathbb{R}^D$,
and any coefficients $a_1,\dots,a_n\in\mathbb{R}$. Then
\begin{align}
    \sum_{i=1}^n\sum_{j=1}^n a_i a_j\,k\!\big(x^{(i)},x^{(j)}\big)
    &=\sum_{i=1}^n\sum_{j=1}^n a_i a_j\,\mathbb{E}\!\big[U^{(m)}(x^{(i)})U^{(m)}(x^{(j)})\big] \\
    &=\mathbb{E}\!\left[\left(\sum_{i=1}^n a_i\,U^{(m)}(x^{(i)})\right)^2\right]
    \ \ge\ 0.
\end{align}
Therefore $k$ is a valid covariance kernel.

\medskip
\noindent\textbf{Conclusion.}
For any $n$ and any inputs $x^{(1)},\dots,x^{(n)}$,
\[
\big(\bar{y}^{(M)}(x^{(1)}),\dots,\bar{y}^{(M)}(x^{(n)})\big)
=
\frac{1}{\sqrt{M}}\sum_{m=1}^M Z^{(m)}
\Rightarrow
\mathcal{N}\!\big(0,\Sigma\big),
\qquad
\Sigma_{ij}=k(x^{(i)},x^{(j)}).
\]
Therefore the finite-dimensional distributions of $\bar{y}^{(M)}$ converge to those of a centered Gaussian process
$g\sim\text{GP}(0,k)$.
\qed

\subsection{Proof of Theorem \ref{thm:postpred_gp}}

\begin{proof}

Set $D:=(X,\tilde y)$ and let $T := X\cup X_* = \{t_i\}_{i=1}^r\subset \mathbb{R}^D$, where $r:=|T|$.

\medskip\noindent
\textbf{Step 1 :Likelihood depends only on finitely many evaluations and is bounded-continuous.}
For the current setting, the likelihood depends on the model only through the finite vector of evaluations
$\bar y^{(M)}(X)\in\mathbb{R}^{n}$. Up to a multiplicative constant independent of $z$, the likelihood as a function of
\[
\ell(z)\;:=\;\exp\!\Big(-\tfrac{1}{2\sigma_n^2}\|\tilde y - z\|_2^2\Big),
\qquad z\in\mathbb{R}^{n}.
\]
Thus $\ell$ is a function of $z=\bar y^{(M)}(X)$ only, it is continuous in $z$, and it is bounded since
$0<\ell(z)\le 1$ for all $z$. In particular, this verifies Assumption 1 of \cite{hron2020exact} for Proposition 1 for the Gaussian likelihood.
Moreover, since $\ell(z)>0$ for all $z$, we also have $\int \ell\, dP>0$ for any probability measure $P$ on $\mathbb{R}^{n}$

\medskip\noindent
\textbf{Step 2 :Reduce to a finite-dimensional prior:}
Define the finite-dimensional random vectors
\[
U^{(M)} := \bar y^{(M)}(T)\in\mathbb{R}^{r},
\qquad
U := g(T)\in\mathbb{R}^{r}.
\]
By Theorem~\ref{thm:prior_gp_main} applied to the finite set of inputs $T$,
\[
U^{(M)} \Rightarrow U,
\qquad
U \sim \mathcal{N}\!\big(0,\,K(T,T)\big).
\]
Where $K(T,T)$ is the $r\times r$ Gram matrix with entries $k(t_a,t_b)$.

\medskip\noindent
\textbf{Step 3 : Continuity of the posterior under weak convergence of the prior.} \\
Let $P_M:=\mathcal{L}(U^{(M)})$ and $P:=\mathcal{L}(U)$ be the prior distributions of
\[
U^{(M)}=\bar y^{(M)}(T)\in\mathbb{R}^r,
\qquad
U=g(T)\in\mathbb{R}^r.
\]
By Step~2 we have $P_M\Rightarrow P$ weakly on $\mathbb{R}^r$.
Let $\pi_X:\mathbb{R}^r\to\mathbb{R}^{|X|}$ denote the coordinate projection selecting the components
corresponding to the training inputs $X\subset T$.

By the definition of weak convergence, it suffices to show that for any bounded continuous
test function $\varphi:\mathbb{R}^r\to\mathbb{R}$,
\[
\int \varphi(u)\,P_M(du\mid D)\;\longrightarrow\;\int \varphi(u)\,P(du\mid D).
\]
By Step~1, the likelihood weight $u\mapsto \ell(\pi_X(u))$ is bounded and continuous on $\mathbb{R}^r$.
Moreover, the posterior $P_M(\cdot\mid D)$ is absolutely continuous with respect to $P_M$ with
Radon--Nikodym derivative
\[
\frac{dP_M(\cdot\mid D)}{dP_M}(u)=\frac{\ell(\pi_X(u))}{Z_M},
\qquad
Z_M:=\int \ell(\pi_X(u))\,P_M(du).
\]
Substituting this density yields
\[
\int \varphi(u)\,P_M(du\mid D)
=
\frac{1}{Z_M}\int \varphi(u)\,\ell(\pi_X(u))\,P_M(du).
\]
Since $\ell\circ\pi_X$ is bounded and continuous and $P_M\Rightarrow P$, we have
\[
Z_M=\int \ell(\pi_X(u))\,P_M(du)\;\longrightarrow\;\int \ell(\pi_X(u))\,P(du)=:Z.
\]
By Step~1 the limit $Z$ is strictly positive. Similarly, the map
$u\mapsto \varphi(u)\,\ell(\pi_X(u))$ is bounded and continuous on $\mathbb{R}^r$, hence
\[
\int \varphi(u)\,\ell(\pi_X(u))\,P_M(du)
\;\longrightarrow\;
\int \varphi(u)\,\ell(\pi_X(u))\,P(du).
\]
Then,

\[
\int \varphi(u)\,P_M(du\mid D)\;\longrightarrow\;\int \varphi(u)\,P(du\mid D),
\]
and therefore $P_M(\cdot\mid D)\Rightarrow P(\cdot\mid D)$ weakly on $\mathbb{R}^r$, that is,
\[
U^{(M)}\mid D \;\Rightarrow\; U\mid D.
\]
This is the finite-dimensional Bayes-continuity argument used in \cite{hron2020exact}.

\medskip\noindent
\textbf{Step 4 : Take the $X_*$ marginal.}
The map $\pi_*:\mathbb{R}^{r}\to\mathbb{R}^{|X_*|}$ that selects the coordinates corresponding to $X_*$
is continuous. Hence, by the continuous mapping theorem applied to the posterior convergence in Step~3,
\[
\pi_*\!\big(U^{(M)}\big)\mid D \Rightarrow \pi_*(U)\mid D,
\]
That is,
\[
\bar y^{(M)}(X_*)\mid D \Rightarrow g(X_*)\mid D,
\]
which is exactly $\bar y_*^{(M)}\mid D \Rightarrow g_*\mid D$.

\textbf{Step 5 : Identify the limit as GP regression.}
Since $g\sim\text{GP}(0,k)$, the joint vector $\big(g(X),g(X_*)\big)$ is multivariate normal, with covariance blocks determined by the kernel $k$.
Together with the Gaussian observation model $\tilde y\mid g(X)\sim\mathcal{N}\!\big(g(X),\sigma_n^2 I_n\big)$, it follows that the conditional
distribution of the latent test values $g(X_*)$ given $D=(X,\tilde y)$ is exactly the standard Gaussian-process regression posterior distribution at $X_*$
\cite{rasmussen2006gpml}.
\end{proof}

\section{Model hyperparameters for Section \ref{subsec:regression_with_UQ}}
\label{appendix:model_hyperparameters}

\subsection{GP, BNN, and B-INN in 1D input space}

Model hyperparameters used to draw Figure \ref{fig:1D_wo_gap_predictions_comparison} are provided below:

\begin{itemize}
\item GP: Observation noise variance is 0.04, signal variance is 1.0, and kernel length scale is 0.5. In standard GP, the hyperparameters can be optimized by minimizing the negative log marginal likelihood. However, we kept using these hyperparameters to make a fair comparison with BNN.

\item B-INN: Observation noise variance is 0.04, signal variance is 1.0, and kernel length scale is 0.5. The number of basis functions was varied: Figure \ref{fig:1D_wo_gap_predictions_comparison}(b) had 60 basis functions centered at training data, while Figure \ref{fig:1D_wo_gap_predictions_comparison}(c) had 20 equispaced basis functions.

\item BNN + Metropolis Hastings: We sampled a total of 6,000 cases, with the first 1,000 being burned out. Step size is 0.01.

\item BNN + Hamiltonian Monte Carlo: We sampled a total of 6,000 cases, with the first 1,000 being burned out. Step size is 0.01. The number of leapfrog steps per iteration is 20.

\item BNN + Variational Inference: We ran 2,000 epochs with a learning rate of 0.01. 

\end{itemize}

\subsection{Open Source AeroDynamics Dataset-BlendedNet}
\begin{itemize}


\item B-INN: Observation noise variance $\sigma_n^2=0.0225$ and signal variance $\sigma_w^2=1.0$ for all models.
We used early stopping with patience 30. The number of modes was set to $M=2$.
The number of basis functions was varied per dimension with center counts
$[25,25,16,10,10,10,12]$.
For each output dimension, we additionally varied the per-dimension kernel length scales:
\[
\begin{aligned}
C_p    &:\ [1.13,\ 1.13,\ 1.50,\ 2.0,\ 2.0,\ 2.0,\ 1.50],\\
C_{f_x}&:\ [1.27,\ 1.27,\ 1.35,\ 2.0,\ 2.0,\ 2.0,\ 1.50],\\
C_{f_y}&:\ [1.169,\ 1.169,\ 1.50,\ 2.0,\ 2.0,\ 2.0,\ 1.50],\\
C_{f_z}&:\ [1.20,\ 1.20,\ 1.35,\ 2.0,\ 2.0,\ 2.0,\ 1.50].
\end{aligned}
\]

\item BNN + Hamiltonian Monte Carlo: We used the No-U-Turn Sampler (NUTS), an adaptive variant of HMC, with 250 warmup steps, 250 post warm up samples, and 0.95 target acceptance probability.

\item BNN + Variational Inference: We ran 3,000 epochs with a learning rate of 0.01. 

\end{itemize}

\section{BlendedNet data preparation}
\label{appendix:blendednet}

The data source of BlendedNet can be found in \cite{sung2025blendednet}. The original data has 19 inputs (three x,y,z coordinates on the airplane surface, three surface normal vectors, nine geometric parameters, and four flight conditions) and 4 outputs (aerodynamic coefficients - Cp, Cfx, Cfy, and Cfz). 

The original dataset is practically too large to train with BNNs. Therefore, we reduce the data size and dimensionality. The reduced dataset has 7 inputs (three for x,y,z coordinates on a plane surface and four flight conditions - altitude, Reynolds number, Mach number, and angle of attack) and 4 outputs (aerodynamic coefficients - Cp, Cfx, Cfy, Cfz). Since we fixed the geometric parameters, the case\_130.vtk to case\_139.vtk files that share the same geometric parameters (denoted as geom\_13) were extracted. Among them, the case\_130.vtk simulaiton data was used for the test set.

\section{Active learning algorithm}
\label{appendix:active_learning_algorithm}

The active learning procedure operates on a candidate parameter pool $\mathcal{C}$ and proceeds for $T$ rounds.
An initial labeled set $\mathcal{L}_0 \subset \mathcal{C}$ and a validation set $\mathcal{V} \subset \mathcal{C}$ are first sampled such that $\mathcal{L}_0 \cap \mathcal{V}=\emptyset$. We then update the remaining candidate pool by
$\mathcal{C} \leftarrow \mathcal{C}\setminus(\mathcal{L}_0 \cup \mathcal{V}),$ and train an initial B-INN surrogate $f_0$ on $\mathcal{L}_0$.

At each round $t \geq 1$, the predictive uncertainty of the current model $f_{t-1}$ is evaluated for every remaining candidate $\theta \in \mathcal{C}$ over a fixed set of spatial acquisition points $\mathcal{X}_{\mathrm{acq}}$. 
Specifically, the pointwise predictive standard deviation $\sigma_{\theta}(x)$ is computed as
\[
\sigma_{\theta}(x) = f_{t-1}.\textsc{PredictStd}(\theta, x), \quad x \in \mathcal{X}_{\mathrm{acq}}.
\]
An acquisition score is then defined by averaging the uncertainty over the spatial domain,
\[
s(\theta) = \frac{1}{|\mathcal{X}_{\mathrm{acq}}|} \sum_{x \in \mathcal{X}_{\mathrm{acq}}} \sigma_{\theta}(x).
\]
The next query point is selected by maximizing this score,
\[
\theta_{\mathrm{new}} = \arg\max_{\theta \in \mathcal{C}} s(\theta),
\]
after which $\theta_{\mathrm{new}}$ is removed from the candidate pool and high-fidelity labeled data $\mathcal{D}_{\mathrm{new}}(\theta_{\mathrm{new}})$ are obtained.

The labeled set is updated according to
\[
\mathcal{L}_t = \mathcal{L}_{t-1} \cup \mathcal{D}_{\mathrm{new}}(\theta_{\mathrm{new}}),
\]
and a new surrogate model $f_t$ is trained using a warm-start initialization based on the posterior mean of $f_{t-1}$. 
Model performance is monitored using the validation error
\[
\mathrm{RMSE}_t = \textsc{RMSE}(f_t, \mathcal{V}).
\]
This iterative process continues until $T$ rounds are completed or the candidate pool $\mathcal{C}$ is exhausted, producing a sequence of increasingly accurate and uncertainty-aware surrogate models $\{f_t\}_{t=0}^T$.

\begin{algorithm}[H]
\caption{Active Learning algorithm for B-INN}
\label{alg:active-learning}
\small
\begin{algorithmic}[1]
\REQUIRE Candidate pool $\mathcal{C}$
\REQUIRE Active Learning rounds $T$
\REQUIRE Spatial points $\mathcal{X}_{\text{acq}}$ (scoring)
\ENSURE Trained models $\{f_t\}_{t=0}^{T}$

\STATE Sample initial training set $\mathcal{L}_0$ and validation set $\mathcal{V}$ from $\mathcal{C}$
\STATE $\mathcal{C} \leftarrow \mathcal{C}\setminus (\mathcal{L}_0 \cup \mathcal{V})$
\STATE Train initial model: $f_0 \leftarrow \texttt{Train}(\mathcal{L}_0)$

\FOR{$t=1$ \textbf{to} $T$}
    \IF{$|\mathcal{C}| = 0$}
        \STATE \textbf{break}
    \ENDIF

    \FORALL{$\theta \in \mathcal{C}$}
        \STATE $\sigma_\theta \leftarrow f_{t-1}.\texttt{predict\_std}(\theta,\mathcal{X}_{\text{acq}})$
        \STATE $s(\theta) \leftarrow \frac{1}{|\mathcal{X}_{\text{acq}}|}
        \sum_{x\in\mathcal{X}_{\text{acq}}} \sigma_\theta(x)$
    \ENDFOR

    \STATE $\theta_{\text{new}} \leftarrow \operatorname*{arg\,max}_{\theta \in \mathcal{C}} s(\theta)$
    \STATE $\mathcal{C} \leftarrow \mathcal{C}\setminus \{\theta_{\text{new}}\}$

    \STATE Sample labeled data at $\theta_{\text{new}}$ and denote it by $\mathcal{D}_{\text{new}}(\theta_{\text{new}})$
    \\  $\mathcal{D}_{\text{new}}(\theta)$ contains labeled samples (e.g., spatial points and targets) at parameter $\theta$.
    \STATE $\mathcal{L}_t \leftarrow \mathcal{L}_{t-1} \cup \mathcal{D}_{\text{new}}(\theta_{\text{new}})$

    \STATE Warm-start: set prior mean from posterior mean of $f_{t-1}$
    \STATE Retrain model: $f_t \leftarrow \texttt{Train}(\mathcal{L}_t)$
    \STATE Compute validation error: $\mathrm{RMSE}_t \leftarrow \texttt{RMSE}(f_t,\mathcal{V})$
\ENDFOR

\STATE \textbf{return} $\{f_t\}_{t=0}^{T}$
\end{algorithmic}
\end{algorithm}

\section{Governing equations for active learning}
\label{appendix:governing_equations_for_active_learning}

\subsection{Poisson equation in 3D space and 1D parameter input space}

Consider a parametric Poisson equation on the unit cube \(\Omega=[0,1]^3\) with homogeneous Dirichlet boundary conditions. For each parameter
\(p\in[0,1]\), we define \(u(\cdot;p)\) as the solution of
\begin{equation}
    \label{eq:poisson_fixed_pde}
    \begin{cases}
        -\Delta u(\boldsymbol{x};p) = f(\boldsymbol{x};p),
        & \boldsymbol{x}\in \Omega,\\[3pt]
        u(\boldsymbol{x};p)=0,
        & \boldsymbol{x}\in \partial\Omega,
    \end{cases}
\end{equation}
where \(\boldsymbol{x}=(x_1,x_2,x_3)\) and \(\Delta\) denotes the Laplacian with respect to \(\boldsymbol{x}\). 

The forcing is a fixed three-mode separable function
\begin{equation}
    \label{eq:poisson_fixed_forcing}
    f(\boldsymbol{x};p)
    =
    \sum_{r=1}^{3} c_r\, p \prod_{j=1}^{3}\sin\!\big(\pi m^{(r)}_j x_j\big),
\end{equation}
with mode vectors and coefficients
\[
    \begin{aligned}
        \boldsymbol{m}^{(1)}&=(2,3,2), &
        \boldsymbol{m}^{(2)}&=(4,1,3),\\
        \boldsymbol{m}^{(3)}&=(5,5,2), &
        (c_1,c_2,c_3)&=(1.0,\ 0.8,\ 1.2).
    \end{aligned}
\]

The closed-form solution to this problem is:
\begin{equation}
\label{eq:poisson_fixed_solution}
u(\boldsymbol{x};p)
=
\sum_{r=1}^{3}
c_r\,
\frac{p}{\pi^2\sum_{j=1}^{3}\big(m^{(r)}_j\big)^2}
\prod_{j=1}^{3}\sin\!\big(\pi m^{(r)}_j x_j\big).
\end{equation}

This yields a fully analytic mapping \((\boldsymbol{x},p)\mapsto u(\boldsymbol{x};p)\), which we use to generate data.

\subsection{Heat equation in 2D space, time, and 2D parameter input space}

The governing parametric PDE of this heat equation can be written as:
\begin{equation}
    \label{eq:heat_pde}
    \left\{
    \begin{array}{l}
        \dot{u}(\boldsymbol{x},t) + k\Delta u(\boldsymbol{x},t) = b(\boldsymbol{x},P)
        \quad \text{in } \Omega_x \otimes \Omega_t,\\[2pt]
        u(\boldsymbol{x},t)\big|_{\partial\Omega_x} = 0,\\[2pt]
        u(\boldsymbol{x},0) = 0,
    \end{array}
    \right.
\end{equation}
defined in a spatial domain $\Omega_x=[0,1]^2$, a temporal domain $\Omega_t=(0,0.04]$, and parametric domains
$\Omega_k=[1,4]$ and $\Omega_P=[100,200]$.

following~\cite{guo2025interpolating}.
The source term is modeled as a sum of Gaussian heat sources
\begin{equation}
\label{eq:heat_rhs}
b(\mathbf{x},P)=\sum_{i=1}^{n_s} P\exp\!\left(-\frac{2\big((x-x_i)^2+(y-y_i)^2\big)}{r_0^2}\right),
\end{equation}
where $r_0=0.05$ is the standard deviation that characterizes the width of the source function and $n_s=16$ source centers $\{(x_i,y_i)\}_{i=1}^{n_s}$ are fixed.
Simulation data are generated by running a finite element analysis (FEA) for different parameter pairs $(k,P)$, producing training
examples of the form $(x,y,t,k,P)\mapsto u(x,y,t;k,P)$.

\section{Active learning and model hyperparameters for Section \ref{subsec:active_learning_with_B-INN}}
\label{appendix:active_learning_and_model_hyperparameters}

\subsection{Poisson equation}

B-INN and BNN Models are initialized with 8 training parameter pairs, then acquire 1 new parameter per round (20 rounds total).

\begin{itemize}
    \item B-INN:Observation noise variance is 0.001, signal variance is 1.0 for all models. The number of Mode was set as 10. The center counts and length scale varied for the parametric dimension. The spatial dimensions had 16 basis functions with a length scale of 1.2, and the parametric dimension had 6 basis function with a length scale of 1.0. We trained the model for 40 iterations.

    \item BNN + Hamiltonian Monte Carlo: We again used the No-U-Turn Sampler (NUTS) with 600 warm up steps, 600 post warm up samples, and 0.85 target acceptance probability. For each Active learning round, an additional 250 warm up steps and 250 post warm up steps were taken.
    
    \item BNN + Variational Inference: We ran 15,000 epochs with a learning rate of 0.001. Every active learning round, we ran the model for an additional 4000 epochs.
\end{itemize}

\subsection{Heat equation}

B-INN and BNN Models are initialized with 35 training parameter pairs, then acquire 1 new parameter per round (20 rounds total).

\begin{itemize}
    \item B-INN:Observation noise variance is 0.001, signal variance is 1.0 for all models. The number of Mode was set as 15. For each input dimension, the center counts was $[16, 16, 16, 12, 12]$ with length scales $[2, 2, 2, 0.9, 8.36]$ 40 iterations were ran every iteration, including initial training.

    \item BNN + Hamiltonian Monte Carlo: We again used the No-U-Turn Sampler (NUTS) with 200 warm up steps, 150 post warm up samples, and 0.85 target acceptance probability. For each Active learning round, an additional 120 warm up steps and 80 post warm up steps were taken.
    
    \item BNN + Variational Inference: We ran 20,000 epochs with a learning rate of 0.001. Every active learning round, we ran the model for an additional 5000 epochs.
\end{itemize}

\end{document}